\documentclass[12pt]{article}

\usepackage{amsmath}
\usepackage{amsfonts}
\usepackage{amsthm}
\usepackage{geometry}
\usepackage{mathrsfs}
\usepackage{amssymb}
\usepackage[percent]{overpic}
\usepackage{bm}
\usepackage[all]{xy}
\usepackage{graphicx}
\usepackage{subfigure}
\usepackage{latexsym}
\usepackage{epigraph}
\usepackage{hyperref}
\usepackage{fancyhdr}
\usepackage{comment}

\usepackage{mathtools}
\mathtoolsset{showonlyrefs}

\numberwithin{equation}{section}
\usepackage{color}

\newtheorem{theorem}{Theorem}[section]

\newtheorem*{thma}{Theorem A}

\newtheorem{proposition}{Proposition}[section]
\newtheorem*{obs}{Observation}

\newtheorem{lemma}{Lemma}[section]
\theoremstyle{remark}
\newtheorem{remark}{Remark}[section]

\newtheorem*{ack}{Acknowledgement}

\def\re{\operatorname{Re}}
\def\Im{\operatorname{Im}}
\def\area{\operatorname{area}}

\def\res{\operatorname{Res}}
\def\logarea{\operatorname{Logarea}}
\def\sing{\operatorname{Sing}}
\def\dens{\operatorname{dens}}

\def\diam{\operatorname{diam}}

\def\dim{\operatorname{dim}}
\def\hc{\operatorname{\widehat{\mathbb C}}}
\def\c{\operatorname{\mathbb C}}
\def\R{\operatorname{\mathbb R}}

\def\ud{\operatorname{\mathbb D}}
\def\dim{\operatorname{dim}}
\def\h{\operatorname{\mathbb{H}}}
\def\j{\operatorname{\mathcal{J}}}
\def\f{\operatorname{\mathcal{F}}}
\def\I{\operatorname{\mathcal{I}}}
\def\b{\operatorname{\mathcal{B}}}
\def\s{\operatorname{\mathcal{S}}}
\def\sing{\operatorname{Sing}}

\begin{document}
\title{Hausdorff dimension of escaping sets of meromorphic functions}
\author{Magnus Aspenberg and Weiwei Cui}
\date{}
\maketitle

\begin{abstract}
We give a complete description of the possible Hausdorff dimensions of escaping sets for meromorphic functions with a finite number of singular values. More precisely, for any given $d\in [0,2]$ we show that there exists such a meromorphic function for which the Hausdorff dimension of the escaping set is equal to $d$. The main ingredient is to glue together suitable meromorphic functions by using quasiconformal mappings. Moreover, we show that there are uncountably many quasiconformally equivalent meromorphic functions for which the escaping sets have different Hausdorff dimensions.



\end{abstract}

\section{Introduction and main results}\label{imr}

In this paper we are considering transcendental meromorphic functions in the plane. If $f$ is such a function, denote by $\f(f)$ the Fatou set of $f$ and $\j(f)$ the Julia set of $f$.  See \cite{bergweiler1} for an introduction and basic results. The fundamental \emph{escaping set} is defined as 
$$\I(f)=\left\{\,z\in\c:\,f^{n}(z)\to\infty\,\text{~as~}\, n\to\infty\, \right\},$$
where $f^n(z)$ is the $n$-th iterate of $z$ under $f$. 
Eremenko first studied this set in the dynamics of transcendental entire functions \cite{eremenko3}. In particular, he proved that $\I(f)$ is non-empty and $\j(f)=\partial\I(f)$. These results were later generalised to transcendental meromorphic functions by Dom\'inguez \cite{dominguez1}.

\smallskip
The escaping set has been explored from various perspectives. We will consider this set from the point of view of the Hausdorff dimension. McMullen proved in \cite{mcmullen11} that the Julia set of the map $f(z) = \sin(\alpha z+\beta)$ has positive Lebesgue measure for $\alpha\neq0$, and that the Hausdorff dimension of the Julia set of the map $f(z)=\lambda e^z$ is equal to two for $\lambda\neq 0$. His result actually holds for the escaping sets. Note that for these functions, $\I(f)\subset\j(f)$; see \cite{eremenko2}. The first result of McMullen was later extended to more general maps; see \cite{aspenberg1, cuiwei1} and also \cite{eremenko2}. Bara\'nski and Schubert independently generalised the second result of McMullen \cite{baranski1, schubert}.  To state their result, first we recall that the singular set $\sing(f^{-1})$ of a meromorphic function $f$ is the set of critical and asymptotic values of $f$. The \emph{Eremenko-Lyubich} class, denoted by $\b$, is defined by
$$\b:=\left\{\,f:\c\to\hc\, \text{transcendental and meromorphic},\,\sing(f^{-1})\cap\c\, \text{is bounded}\, \right\}.$$
We call these functions \emph{Eremenko-Lyubich functions}, which received much interest in transcendental dynamics recently. See \cite{sixsmith6} for a survey of dynamics of entire functions in the class $\b$. The result of Bara\'nski and Schubert mentioned above can be stated as follows: If $f\in\b$ is entire and of finite order, then $\I(f)$ has Hausdorff dimension two. Their argument can actually be used to show that the same conclusion holds if $f\in\b$ is meromorphic and of finite order for which $\infty$ \emph{is} an asymptotic value. To see this, one can consider the set of points escaping to $\infty$ in the logarithmic tracts over $\infty$ and compute the Hausdorff dimension of this set by using the argument of Bara\'nski and Schubert to get the conclusion. In contrast with this, Bergweiler and Kotus proved that the escaping set may have Hausdorff dimension strictly less than two if $\infty$ is \emph{not} an asymptotic value  \cite{bergweiler2}. This happens, in particular, if the multiplicities of poles are uniformly bounded above.  Moreover, they even give a complete characterisation for the Hausdorff dimensions of escaping sets for Eremenko-Lyubich functions; see \cite[Theorem 1.2]{bergweiler2}. Here and in the following, $\dim A$ stands for the Hausdorff dimension of a set $A$. 

\begin{thma}[Bergweiler, Kotus]\label{thma}
$$\left\{\,\dim\I(f):\,f\in\b\,\right\}=[0,2].$$
\end{thma}

\medskip
Another class of functions which has attracted much attention is the so-called \emph{Speiser class},
$$\s:=\left\{\,f:\c\to\hc\, \text{transcendental and meromorphic},\,\sing(f^{-1})\, \text{is finite} \,\right\}.$$
Functions in this class are called \emph{Speiser functions}. This is a more restrictive class than the class $\b$. In some sense, Speiser functions stand in between rational functions and general meromorphic functions, and thus have attracted attention from, for instance, Nevanlinna theory \cite{teichmuller2} and also transcendental dynamics \cite{eremenko2}. A good understanding of Speiser functions will give some insight into the understanding of general meromorphic functions. However, there are striking differences between these two classes. In case of entire functions, the significant difference is addressed recently by Bishop in \cite{bishop2, bishop5}. 

The main purpose of this paper is to compare these two classes in the meromorphic setting in terms of the Hausdorff dimension of their escaping sets. To be more specific, our intention is to show that in the above Theorem A, functions can actually be taken in the smaller class $\s$.

\begin{theorem}\label{tm}
	$$\left\{\,\dim\I(f):\,f\in\s\,\right\}=[0,2].$$
\end{theorem}

This gives a complete description of the possible Hausdorff dimensions of the escaping sets of Speiser functions, which also strengthens the above Theorem A. To prove our result, we will construct Speiser functions by using quasiconformal mappings. This is quite different from the method used in \cite{bergweiler2}, in which the authors proved Theorem A by considering suitable infinite sums.

As shown by \cite[Theorem 1.1]{bergweiler2}, for $f\in\b$ of finite order for which $\infty$ is not an asymptotic value, if the multiplicities of poles are uniformly bounded, the Hausdorff dimension of escaping sets will in some sense depend on the order of the function. In particular, to obtain escaping sets of very small Hausdorff dimensions, say, close to zero, one will need Speiser functions of very small orders. However, the Denjoy-Carleman-Ahlfors theorem tells that, if a Speiser function has an asymptotic value, the (lower) order of the function is at least $1/2$; see \cite{goldbergmero}. Therefore, to achieve small orders, the desired Speiser functions cannot have any asymptotic values. This will in fact be a common property for all Speiser functions we construct in this paper. To be more precise, to prove Theorem \ref{tm} we will need to construct functions with properties shown in the following theorem.

\begin{theorem}\label{thm1}
	For any given number $d\in [0,2)$, there exists a meromorphic function $f\in\s$ satisfying the following properties.
	\begin{itemize}
	\item[$(1)$] $f$ has no asymptotic values in $\hc$.
	\item[$(2)$] All poles of $f$ have multiplicity $2$.
	\item[$(3)$] The order of $f$ is $\dfrac{d}{2-d}$;
	\item[$(4)$] $\dim\I(f)=d$.
	\end{itemize}
\end{theorem}

Notice that the above theorem gives escaping sets of Hausdorff dimension strictly less than $2$. To achieve a Speiser function with a full dimensional escaping set, by \cite[Theorem 1.1]{bergweiler2}, one can consider certain functions with infinite order or functions with finite order and unbounded multiplicities of poles. Here we provide a function with the former property. See Section \ref{efd} for a discussion on other functions whose escaping sets also have Hausdorff dimension two. 

\begin{proposition}\label{full}
Let $\wp$ be a Weierstra{\ss} elliptic function with respect to a certain lattice and let $c$ be a number chosen such that it is not a pole of $\wp$. Let $f(z)=\wp(e^z+c)$. Then $\dim\I(f)=2.$
\end{proposition}

The proof of Theorem \ref{thm1} together with Proposition \ref{full} is provided in section \ref{eses}.

\bigskip
Our second result will deal with the question concerning the invariance of the Hausdorff dimensions of escaping sets in the parameter space. More precisely, two meromorphic functions $f$ and $g$ are \emph{quasiconformally equivalent} if there exist quasiconformal mappings $\varphi,\,\psi$ of the plane such that $\varphi\circ f=g\circ \psi$; see \cite[Section 3]{eremenko2}. (They are \emph{topologically equivalent} if instead of quasiconformal mappings, one takes only homeomorphisms.) Functions satisfying this equivalence relation are considered to belong to the same parameter space. (In case of Speiser functions, this parameter space is a finite dimensional complex manifold.)

Then the question mentioned above can be stated as follows: Let $f, g\in\b$ be quasiconformally equivalent, do their escaping sets have the same Hausdorff dimension? The question was mentioned in \cite{rempe11} and in \cite{bergweiler10} for transcendental entire functions. Moreover, it is proved in \cite{rempe11} that if two such entire functions are affinely equivalent, then their escaping sets have the same Hausdorff dimension. 

There are quasiconformally but not affinely equivalent transcendental entire functions in class $\b$ whose escaping sets have the same Hausdorff dimension. Such examples were constructed in \cite{epstein5,bishop6}. The escaping sets of these functions have the same Hausdorff dimension by the result of Bara\'nski and Schubert mentioned before, but they are not affinely equivalent since they have different orders. In the meromorphic setting, some positive results towards the question are also known. For example, Ga\l{}azka and Kotus proved that the Hausdorff dimension of escaping sets of some simply periodic functions and all doubly periodic functions depend only on the multiplicities of poles, which cannot be changed under equivalence relation and thus the dimension is invariant \cite{kotus2, kotus6}. Also in \cite{cuiwei2}, meromorphic functions with rational Schwarzian derivatives are shown to have the invariance property. However, despite of these results our construction of functions in Theorem \ref{thm1} above gives profound counterexamples even in the Speiser class.

\begin{theorem}\label{thm3}
There exist uncountably many quasiconformally equivalent Speiser meromorphic functions whose Hausdorff dimensions of escaping sets are different.
\end{theorem}

The proof of this result depends essentially on the non-invariance of orders under the above equivalence relation: Quasiconformally equivalent finite-order meromorphic Speiser functions may have different orders of growth. While  it was known already for meromorphic functions (see discussion in \cite{epstein5}), it is only shown quite recently that this also holds in the entire setting \cite{bishop6}.

\medskip
\noindent{\emph{Structure of the article}.} In Section \ref{pre} we give some preliminaries that will be used for our construction and also for the estimate of Hausdorff dimensions. Section \ref{CS} is devoted to the construction of desired Speiser meromorphic functions by using a quasiconformal surgery. In Section \ref{eses} we estimate the Hausdorff dimensions of escaping sets of the functions constructed in Section \ref{CS}, which will complete the proof of Theorem \ref{thm1} and also Theorem \ref{thm3}.

\begin{ack}
We would like to thank Walter Bergweiler for many useful comments, in particular for observing 
that the construction in Section \ref{sec3.3} could be substantially simplified by using the construction in Section \ref{sec3.2} and also for suggesting the example in Proposition \ref{full}. The second author acknowledges partial support from the China Postdoctoral Science Foundation (No.2019M651329). We also want to express our gratitude to the Centre for Mathematical Sciences at Lund University for providing a nice working environment. The authors are grateful to the referee for many helpful suggestions and comments.
\end{ack}

\section{Some Preliminaries}\label{pre}

Let $f:\c\to\hc$ be transcendental and meromorphic. A point $c$ is a critical point of $f$ if $f$ has vanishing spherical derivative at $c$. The image of a critical point is called a \emph{critical value}. We say that $a\in\hc$ is an \emph{asymptotic value} of $f$, if there exists a curve $\gamma: (0,\infty)\to\hc$ tending to $\infty$ and $f(\gamma(t))$ tends to $a$ as $t\to\infty$.  As mentioned in the introduction, the \emph{singular set} $\sing(f^{-1})$ of $f$ is the set of all critical and asymptotic values of $f$. Note that we are not excluding the possibility that $\infty$ might be a singular value. The dynamical behaviours of meromorphic functions are closely related to the dynamical behaviours of singular values and thus the singular set has been studied a lot from this respect. We refer to \cite{bergweiler1} for more details on this.

In this section we give some notations that will be used later and also present some preliminary results concerning Nevanlinna theory and quasiconformal mappings.

\smallskip
The upper and lower half planes are denoted by $\h^{+}$ and $\h^{-}$ respectively. We let $\mathbb{Z}$ and $\mathbb{N}$ denote the integers and natural numbers respectively. $\R$ denotes the real axis. $\lfloor x \rfloor$ will denote the integer part of a real number $x$. We will also use disks in terms of Euclidean and spherical metrics. More precisely, for $a\in\c,\,r>0$ and $A\subset \c$, we denote by $\overline{D}(a,r),\,D(a,r),\,\diam A,\,\area A$ the closed and open disk of radius $r$ centred at $a$, the diameter and the area of $A$ respectively. The unit disk is usually denoted by $\ud$. We let $D_{\chi}(a,r),\,\diam_{\chi}A,\,\area_{\chi} A$ be the spherical versions (with the spherical metric). The density of $A$ in $B$, for measurable sets $A$ and $B$, are defined by
$$\dens(A,B)=\frac{\area (A\cap B)}{\area B},\,~\,\dens_{\chi}(A,B)=\frac{\area_{\chi} (A\cap B)}{\area_{\chi} B}.$$

We first recall some notions and preliminary results from the Nevanlinna theory. For more details, we refer to \cite{goldbergmero, hayman1, nevanlinna6}.

Let $n(r, f, a)$ denote the number of $a$-points of $f$ in the closed disk $\overline{D}(0, r)$; that is, the number of solutions of $f(z)=a$ in $\overline{D}(0, r)$. In particular, if $a=\infty$, we will write $n(r,f, \infty)=n(r,f)$ for simplicity. The \emph{integrated counting function} and the \emph{proximity function} are defined respectively as
$$N(r,f)=\int_{0}^{r}\frac{n(t,f)-n(0,f)}{t}\,dt+n(0,f)\log r$$
and
$$m(r,f)=\frac{1}{2\pi}\int_{0}^{2\pi}\log^{+}\left|f\left(re^{i\theta}\right)\right|d\theta,$$
where $\log^{+}a=\max\{0, \log a\}$ for $a>0$. Then the \emph{Nevanlinna characteristic} is defined by
$$T(r,f)=m(r,f)+N(r,f).$$

The following theorem is well known in the value distribution theory of meromorphic functions and is often referred as \emph{the first fundamental theorem of Nevanlinna theory}. 
\begin{theorem}\label{fft}
Let $a\in\c$. Then
$$T(r,f)=T\left(r,\frac{1}{f-a}\right)+\mathcal{O}(1).$$
\end{theorem}

With the aid of the Nevanlinna characteristic function, the \emph{order} and \emph{lower order} of growth of $f$ are defined respectively as
$$\rho(f)=\limsup_{r\to\infty}\frac{\log T(r,f)}{\log r}$$
and
$$\mu(f)=\liminf_{r\to\infty}\frac{\log T(r,f)}{\log r}.$$

In case that $f$ is entire, the term $T(r,f)$ may be replaced by $\log M(r,f)$, where $M(r,f)$ is the maximum modulus, i.e 
$M(r,f) = \max\limits_{|z| = r}|f(z)|$. 

\medskip
We will also need the following result by Teichm\"uller \cite{teichmuller2}, which reduces the calculation of the order to estimating the counting function for poles for certain meromorphic functions. The following version can be found in \cite[Proposition 7.1]{bergweiler3}.

\begin{theorem}[Teichm\"uller]\label{tec}
Let $f\in\b$ be transcendental and meromorphic. Suppose that there exists an $N\in\mathbb{N}$ such that the poles of $f$ have multiplicity at most $N$. If $\liminf_{r\to\infty}m(r,f)/T(r,f)>0$ or, more generally, if $m(r,f)$ is unbounded, then $f$ has a logarithmic singularity over infinity.
\end{theorem}

One consequence of this result is the following lemma.

\begin{lemma}\label{tlemmma}
Let $f\in\b$ be transcendental and meromorphic for which $\infty$ is not an asymptotic value. Then
\begin{equation}\label{tlem}
\rho(f)=\limsup_{n\to\infty}\frac{\log N(r,f)}{\log r}=\limsup_{n\to\infty}\frac{\log n(r,f)}{\log r}.
\end{equation}
\end{lemma}
\begin{proof}
By Theorem \ref{tec} one has that $m(r,f)$ is bounded. So we have
$$T(r,f)=N(r,f)+\mathcal{O}(1),$$
which, by the definition of the order, leads to the first equality of \eqref{tlem}. That the second equality holds is a classical result; see \cite[Chapter 2]{goldbergmero} for more details.
\end{proof}

\medskip
\noindent{\emph{Asymptotic conformality.}} A homeomorphism $\varphi: \c\to\c$ is \emph{quasiconformal} if $\varphi$ is absolutely continuous on almost all horizontal and vertical lines, and moreover, the partial derivatives satisfy
$$|\varphi_{\bar{z}}|\leq k |\varphi_z|$$
almost everywhere for $0<k\leq 1$. Here $\varphi_{\bar{z}}=\varphi_{x}+i\varphi_y$ and $\varphi_{z}=\varphi_x-i\varphi_y$. The \emph{dilatation} of $\varphi$ at a point $z$ is
$$K_{\varphi}(z)=\frac{|\varphi_z|+|\varphi_{\bar{z}}|}{|\varphi_z|-|\varphi_{\bar{z}}|}=\frac{1+|\mu_{\varphi}|}{1-|\mu_{\varphi}|},$$
where $\mu_{\varphi}=\varphi_{\bar{z}}/\varphi_{z}$ is the complex dilatation of $\varphi$. The constant $K=(1+k)/(1-k)$ is called the quasiconformal constant of $\varphi$. The well known measurable Riemann mapping theorem (see \cite{ahlfors8, lehto1}) says that for any given measurable $|\mu|\leq k<1$, there exists a quasiconformal homeomorphism $\varphi:\c\to\c$ such that $\mu_{\varphi}=\mu$ almost everywhere. We refer to \cite{ahlfors8}, \cite{lehto1} for a detailed account of quasiconformal mappings.

A map $g: \c\to\hc$ is \emph{quasi-meromorphic} if it can be written as $g=f\circ\varphi$, where $f:\c\to\hc$ is meromorphic and $\varphi:\c\to\c$ is quasiconformal. Our construction will first give such a quasi-meromorphic function $g$ in the plane and thus a meromorphic function $f$. To obtain desired properties of $f$, we need to control asymptotic behaviours of $\varphi$ near $\infty$ in the relation $g=f\circ\varphi$. This is ensured by the following theorem, which gives a condition for the conformality of a quasiconformal mapping at a point; see \cite{lehto1}.

\begin{theorem}[Teichm\"uller-Wittich-Belinskii's theorem]\label{twb}
Let $\varphi:\c\to\c$ be a quasiconformal mapping and let $K_{\varphi}$ be its dilatation. Suppose that
$$\iint_{|z|>1}\frac{K_{\varphi}(z)-1}{x^2+y^2}\,dxdy<\infty.$$
Then
$$\varphi(z)\sim z~\,\,\text{as}\,\,~~z\to\infty.$$
\end{theorem}

The \emph{logarithmic area} of a set $A\subset \mathbb{R}^2$ is defined as
$$\logarea(A)=\iint_{A}\frac{dxdy}{x^2+y^2}.$$
The following result follows from the above theorem.
\begin{lemma}\label{logareaes}
Let $\varphi:\c\to\c$ be a quasiconformal mapping. Let $A$ be the supporting set of $\varphi$; i.e., the set of points for which $\varphi$ is not conformal. If $\logarea(A\setminus\ud)<\infty$. Then
$$\varphi(z)\sim z~\,\,\text{as}\,\,~~z\to\infty.$$
Upon normalisation, one may assume that $\varphi(z)=z+o(z)$ for large $z$.
\end{lemma}
\begin{proof}
Since $|\mu_{\varphi}|\leq 1$, one has $K_{\varphi}-1=2|\mu_{\varphi}|/(1+|\mu_{\varphi}|)\leq 1$. Now the conclusion follows from the above Theorem \ref{twb}.
\end{proof}

Therefore, in order to obtain the conclusion of the above theorem, it suffices in this paper to check that the supporting set of $\varphi$ outside of the unit disk has finite logarithmic area.
\begin{remark}
The above lemma is sufficient for our later use. We note that many results are focused on the estimate of the error term of the map $\varphi$ near $\infty$.
\end{remark}

\medskip

\noindent{\emph{McMullen's result.}} The estimate of the lower bound of the escaping set will use a result of McMullen \cite{mcmullen11}, which is usually stated using the standard Euclidean metric. The following version, using spherical metric, can be found in \cite{bergweiler2}. To state this result, assume that $E_{\ell}$ is a collection of disjoint compact subsets of $\hc$ for each $\ell\in\mathbb{N}$ satisfying
\begin{itemize}
\item[(i)] each element of $E_{\ell+1}$ is contained in a unique element of $E_{\ell}$;
\item[(ii)] each element of $E_{\ell}$ contains at leat one element of $E_{\ell+1}$.
\end{itemize}
Suppose that $\overline{E}_{\ell}$ is the union of all elements of $E_{\ell}$. Let $E=\bigcap \overline{E}_{\ell}$. Suppose that for $V\in E_{\ell}$,
$$\dens_{\chi}(\overline{E}_{\ell+1}, V)\geq \Delta_{\ell}$$
and
$$\diam_{\chi} V \leq d_\ell$$
for two sequences of positive real numbers $(\Delta_{\ell})$ and $(d_{\ell})$. Then we have the following estimate.
\begin{theorem}\label{mcmlower}
Let $E_{\ell},\,E,\,\Delta_{\ell}$ and $d_{\ell}$ be as above. Then
$$\dim E \geq 2-\limsup_{\ell\to\infty}\frac{\sum_{j=1}^{\ell+1}|\log\Delta_j|}{|\log d_{\ell}|}.$$
\end{theorem}

Finally, we will need the following variation of the well known \emph{Koebe one-quarter theorem and the Koebe distortion theorems}; see \cite{pommerenke1}. 
\begin{theorem}[Koebe's theorem]\label{koebe}
Let $f$ be a univalent function in $D(z_0,r)$ and let $0<\lambda<1$. Then
\begin{equation}\label{estimate of value}
\dfrac{\lambda}{\left(1+\lambda \right)^2}\left|f'(z_0)\right|\leq \left| \dfrac{f(z)-f(z_0)}{z-z_0}\right| \leq \dfrac{\lambda}{\left(1-\lambda \right)^2}\left|f'(z_0)\right|
\end{equation}
and
\begin{equation}\label{estimate of derivative}
\dfrac{1-\lambda}{\left(1+\lambda \right)^3}\left|f'(z_0)\right|\leq \left| f'(z)\right| \leq \dfrac{1+\lambda}{\left(1-\lambda \right)^3}\left|f'(z_0)\right|
\end{equation}
for $|z-z_0|\leq \lambda r$. Moreover,
\begin{equation}\label{one-quarter theorem}
f\left(D(z_0,r)\right)\supset D\left( f(z_0), \frac{1}{4}|f'(z_0)|r  \right).
\end{equation}
\end{theorem}

\section{Construction of Speiser functions}\label{CS}

We construct meromorphic functions in the Speiser class with any prescribed finite order of growth. Since the construction is explicit, the "regularity" of distribution of poles enables us to compute the exact values of Hausdorff dimensions of the escaping sets for these functions. We will separate our construction into three cases. The first case deals with the construction of functions whose orders are equal to $2$ in Section \ref{qc-gluing}. The method will then be used to construct functions  whose orders lie in $(0,2)$ in Section \ref{sec3.2} and functions with orders in $(2,\infty)$ in Section \ref{sec3.3}.

\medskip
Before construction, first we recall the definition of Weierstra{\ss} $\wp$-functions. For $w_i\in\c$ ($i=1,2$) such that $w_1$ is not a real multiple of $w_2$, consider the following lattice defined as
$$\Lambda=\left\{\,m\,w_1+n\,w_2:\, m, n\in\mathbb{Z}\,\right\}.$$
Then the the Weierstra{\ss} elliptic function with respect to the lattice $\Lambda$ is defined by
$$\wp_{\Lambda}(z)=\frac{1}{z^2}+\sum_{w\in\Lambda\setminus\{0\}}\left(\frac{1}{(z-w)^2}-\frac{1}{w^2}  \right).$$
Weierstra{\ss} $\wp$-functions belong to the Speiser class $\s$. More precisely, it has no asymptotic values and four critical values which are
$$\wp_{\Lambda}\left(\frac{w_1}{2}\right),\, \wp_{\Lambda}\left(\frac{w_2}{2}\right),\, \wp_{\Lambda}\left(\frac{w_1+w_2}{2}\right)~\text{~and~}~ \wp_{\Lambda}(0)=\infty.$$
Every lattice point is a double pole. For simplicity, we will write $\wp$ instead of $\wp_{\Lambda}$ if the lattice is clear from the context.

To prove Theorem \ref{thm3}, we will need to show that Speiser functions constructed below are quasiconformally equivalent. In general, it is difficult to tell whether or not two given functions are equivalent. One result that could be useful is a theorem of Teichm\"uller (\cite{teichmuller3}), saying that Speiser functions with the same "combinatorial structure" are quasiconformally equivalent. However, we will take another route, since our construction is explicit and thus enables us to show the equivalence directly.

\smallskip
The following observation for Weierstra{\ss} elliptic functions is folklore for experts in the field and the idea of proof will also be useful later on. We outline a proof below.

\begin{obs}[All $\wp$-functions are equivalent]\label{obsproof}
All Weierstra{\ss} elliptic functions are quasiconformally equivalent.
\end{obs}

\begin{proof}[Proof outline]
Let $\wp_1$ and $\wp_2$ be two Weierstra{\ss} elliptic functions with respect to corresponding lattices
$$\Lambda_1=\left\{m\tau_1+n\tau_2:\,m,\,n\in\mathbb{Z} \right\},$$
$$\Lambda_2=\left\{m\omega_1+n\omega_2:\,m,\,n\in\mathbb{Z} \right\}.$$
Suppose that the finite critical values of $\wp_i$ are $e_{j}^{i}$ for $i=1,\,2$ and $j=1,\,2,\,3$. Denoted by $A_1$ be the parallelogram formed by $\{0,\tau_1/2,\tau_2/2, (\tau_1+\tau_2)/2\}$, and by $A_2$ the parallelogram formed by $\{0,\omega_1/2,\omega_2/2,(\omega_1+\omega_2)/2\}$. Note that $A_1$ is conformally mapped by $\wp_1$ onto a domain $A'_1=\wp_1(A_1)$ bounded by a curve passing through $e_{j}^{1}$ for all $j$ and $\infty$. The same is true for $A_2$. Put $A'_2=\wp_2(A_2)$. Then there exists a quasiconformal homeomorphism $\varphi: A'_1\to A'_2$ fixing $\infty$ and sending $e_{j}^{1}$ to $e_{j}^{2}$ for $j=1,2,3$.

Now let $z\in A_1$. Then we can define a quasiconformal homeomorphism $\psi:A_1\to A_2$ by sending $z$ to $\wp_{2}^{-1}(\varphi(\wp_1(z)))$. By reflection, one can extend $\varphi$ to the whole sphere as a quasiconformal homeomorphism sending all critical values of $\wp_1$ to those of $\wp_2$, and thus can be used to extend $\psi$ to the whole plane. We still use $\varphi$ and $\psi$ as their extensions. By definition of $\psi$, we clearly have
$$\wp_{2}\circ\psi=\varphi\circ\wp_1,$$
which is the quasiconformal equivalence as claimed.
\end{proof}

\subsection{A quasiconformal surgery -- functions of order two} \label{qc-gluing}

We show here how to glue two Weierstra{\ss} $\wp$-functions using quasiconformal mappings to obtain Speiser functions of order $2$. In our later constructions, we will use this technique several times.

Let $\wp_1$ and $\wp_2$ be two Weierstra{\ss} $\wp$-functions whose lattices are respectively given by
$$\Lambda_1=\left\{m+n\tau_1:\,m,\,n\in\mathbb{Z} \right\},$$
$$\Lambda_2=\left\{m+n\tau_2:\,m,\,n\in\mathbb{Z} \right\}.$$
So they have a common period, $1$,  along the real axis. We assume, without loss of generality, that $\Im(\tau_1), \Im(\tau_2) > 0$. Roughly speaking, we consider $\wp_1$ on the upper half-plane $\h^{+}$ and $\wp_2$ on the lower half-plane $\h^{-}$, and glue them together along the real axis $\mathbb{R}$ using a quasiconformal surgery.  To make this argument work, we have to slightly modify Weierstra{\ss} $\wp$-functions. Then the surgery will take place in a horizontal strip around the real axis. The finiteness of the logarithmic area of this strip will give us sufficient control over the asymptotic behaviours of the constructed meromorphic function by the theorem of Teichm\"uller, Wittich and Belinskii mentioned in the last section.

\medskip
Put
\begin{equation}\label{symbol1}
v_1=\wp_1(0),\,\,v_2=\wp_1\left(\frac{1}{2}\right),\,\,v_3=\wp_1\left(\frac{\tau_1+1}{2}\right),\,\,v_4=\wp_1\left(\frac{\tau_1}{2}\right)
\end{equation}
and
\begin{equation}\label{symbol2}
w_1=\wp_2(0),\,\,w_2=\wp_2\left(\frac{1}{2}\right),\,\,w_3=\wp_2\left(\frac{\tau_2+1}{2}\right),\,\,w_4=\wp_2\left(\frac{\tau_2}{2}\right).
\end{equation}
Thus $v_1=w_1=\infty$. We start from the following result.

\begin{proposition}\label{prop1}
There exists an analytic closed curve $\gamma$ in the plane separating the critical values $v_3, v_4, w_3$ and $w_4$ from all other critical values of $\wp_1$ and $\wp_2$ such that, for each $i$, there exists an unbounded analytic curve $\beta_i$ such that $\wp_i(\beta_i)=\gamma$. Moreover, $\beta_i$ is periodic with period $1$ (i.e., $z\in\beta_i$ implies $z+1\in\beta_i$).
\end{proposition}

\begin{proof}
Let $\Gamma$ be a closed Jordan curve on the Riemann sphere $\hc$ passing through all critical values $v_i$ and $w_i$ of $\wp_1$ and $\wp_2$ for $i=1,\dots,4$. Moreover, $\Gamma$ passes through each critical value exactly once, and we assume that the order of critical values for each $\wp$-function is cyclic modulo their index. When one goes along $\Gamma$, the index one meets for each $\wp_i$ is either increasing or decreasing. We also require that $\Gamma$ (when going in one direction in the finite plane) first passes through $v_2$ and $w_2$, and then the rest. $\Gamma$ decomposes the sphere into two Jordan domains, say, $A$ and $B$.

Consider the boundary $L_1$ of the parallelogram $P_1$ consisting of vertices at $0, 1/2, (1+\tau_1)/2$ and $\tau_1/2$. Then $\gamma_1 = \wp_1(L_1)$ is a simple closed curve on $\hc$. Suppose that the two components of $\hc\setminus\gamma_1$ are $A_1$ and $B_1$. Then $P_1$ is mapped conformally onto one of them, say $A_1$. Note that $\Gamma$ is homotopic to $\gamma_1$ relative to $v_1,\,v_2,\,v_3$ and $v_4$. Then there is a quadrilateral $P$ passing through $0, 1/2, (1+\tau_1)/2$ and $\tau_1/2$, which is a deformation of the parallelogram $P_1$ relative to all critical points of $\wp_1$ and is mapped conformally onto $A$. We should also mention that the orders of critical values on the boundary of $A$ and $A_1$ are the same.

The translated parallelogram $Q_1$ with vertices at $1/2, (1+\tau_1)/2, 1+\tau_1/2$ and $1$ is mapped under $\wp_1$ conformally onto $B_1$. And in the same way as above there is another quadrilateral $Q$ passing the vertices of $Q_1$ which is mapped conformally onto $B$. See Figure \ref{claimcon_1}. The periodicity of $\wp_1$ implies that one can partition the whole plane into quadrilaterals constructed above. In other words, $\wp_{1}^{-1}(\Gamma)$ is a deformed graph of $\wp_{1}^{-1}(\gamma_1)$ (which is a graph whose edges are straight segments).

Using the same analysis above to $\wp_2$, we are also able to define a partition of the plane into quadrilaterals induced by considering $\wp_{2}^{-1}(\Gamma)$, which is a graph homotopic to $\wp_{2}^{-1}(\gamma_2)$. Here $\gamma_2 = \wp_2(L_2)$ and $L_2$ is the corresponding boundary of the parallelogram $P_2$ with vertices at $0, 1/2, (1+\tau_2)/2$ and $\tau_2/2$. 

Now let $\gamma$ be a non self-intersecting closed analytic curve in the plane that separates critical values $v_3, v_4$ and $w_3, w_4$ from all other critical values of $\wp_i$ for $i=1,2$. Moreover, we require that $\gamma$ intersects with $\Gamma$ at exactly two points. This implies that $\Gamma$ cuts the curve $\gamma$ into two components such that one of them lies in $A$  and the other lies in $B$.

\begin{figure}[h] 
	\centering
	\includegraphics[width=12cm]{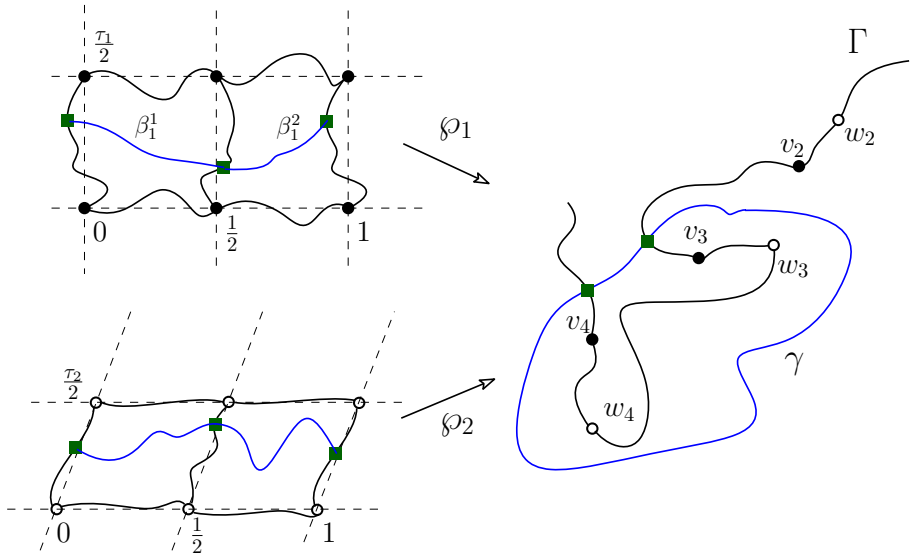}
	\caption{$\Gamma$ is a Jordan curve passing through critical values of $\wp_i$ in cyclic order. For each $i$,  every component of $\c\setminus\wp_{i}^{-1}(\Gamma)$ is a quadrilateral connecting four critical points of $\wp_i$. Dashed lines are original lattice of $\wp_i$, while black solid lines are $\wp_{i}^{-1}(\Gamma)$.}
	\label{claimcon_1}
\end{figure}

In the following we construct the analytic curve $\beta_1$ which is periodic with period one and satisfies $\wp_1(\beta_1)=\gamma$. This follows from the fundamental relations between quadrilaterals induced by $\Gamma$ and $A$ and $B$ established above: each quadrilateral is mapped conformally onto one of $A$ and $B$. We construct $\beta_1$ piece by piece. That $\wp_1: P\to A$ is conformal ensures that there is an analytic curve, denoted by $\beta_{1}^{1}$ in $P$, which is mapped to $\gamma\cap A$. Similarly, there is an analytic curve, say $\beta_{1}^{2}$, in $Q$, which is mapped by $\wp_1$ to $\gamma\cap B$. Then since $\gamma$ does not pass through any critical values of $\wp_1$, the curves $\beta_{1}^1$ and $\beta_{1}^{2}$ share a common endpoint which lies on the common boundary of $P$ and $Q$, which is an edge connecting $1/2$ and $(\tau_1 +1)/2$. Then by periodicity of $\wp_1$, we can extend $\beta_{1}^{1}\cup\beta_{1}^{2}$ periodically, with period $1$, along the horizontal direction. We denote by $\beta_1$ the extended curve. Therefore, $\beta_1$ satisfies our desired properties: it is an analytic curve which is periodic of period one and $\wp_1(\beta_1)=\gamma$.

Exactly the same argument applies to the construction of an analytic curve $\beta_2$, so we omit details and only draw the conclusion that $\beta_2$ is as required.
\end{proof}

\bigskip
Two domains thereby arise from the constructions of $\beta_i$: one in a upper half-plane bounded below by the curve $\beta_1$, and another one in a lower half-plane bounded above by $\beta_2$. Let these two domains be denoted by $\tilde{H}_1$ and $\tilde{H}_2$ respectively. Moreover, denote by $c_i$ one of the intersection points of $\beta_i$ with the imaginary axis. (It could be that each $\beta_i$ intersects the imaginary axis at more than one point and it suffices here to choose any one of them.) We also put
$$H_1=\left\{z: \Im(z)>\Im(c_1)\right\},$$
$$H_2=\left\{z: \Im(z)<\Im(c_2)\right\}.$$

\begin{figure}[h]
	\centering
	\includegraphics[width=13cm]{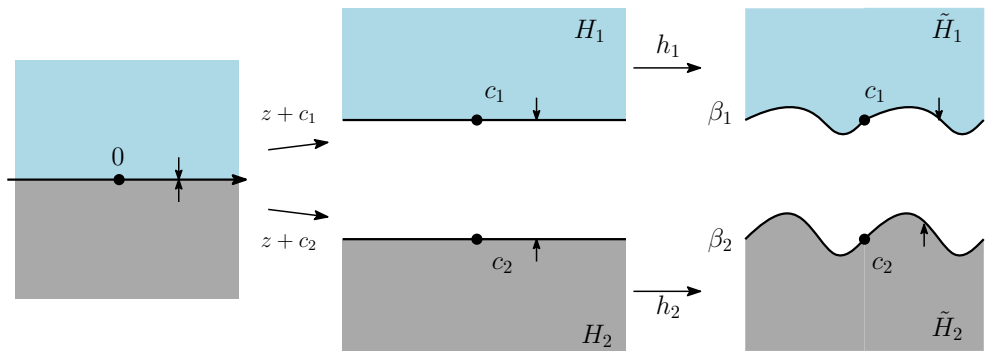}
	\caption{Idea of construction for Speiser functions of order two. The obstruction arises from the discontinuity when one tries to paste together two Weierstra{\ss} elliptic functions, which is then overcome by using a quasiconformal surgery.}
	\label{idea}
\end{figure}

We will need to construct quasiconformal mappings $h_1$ and $h_2$ sending $H_1$ and $H_2$ to $\tilde{H}_1$ and $\tilde{H}_2$ correspondingly such that 
\begin{equation}\label{aim}
\wp_1(h_1(x+c_1))=\wp_2(h_2(x+c_2))\,~\,\text{whenever}\,\, x\in\mathbb{R}.
\end{equation}
See Figure \ref{idea} for a sketch of the idea of construction. For simplicity, we only focus ourselves on the construction of $h_2$ . The construction of $h_1$ goes in the same manner and so we omit most of the details if possible.  

\smallskip
We now proceed with the construction of a quasiconformal mapping 
\[
 h_2: H_2 \rightarrow \tilde{H}_{2}.
\]
This map will be defined piecewise. Moreover, it will be quasiconformal only in a horizontal strip and conformal elsewhere. First we notice that the periodicity of the curve $\beta_2$ implies that one can choose $a$ such that
$$a<\min_{z\in\beta_2}\Im (z).$$
We put
$$l_1=\left\{\,z\in \tilde{H}_2:\,\Im z=a\, \right\}.$$
Moreover, the domain in $\tilde{H}_2$ bounded by $\beta_2$ and $l_1$ is denoted by $S_1$.

We first show that the following holds.
\begin{lemma}\label{pecon}
 There exists $a'<\min_{z\in\beta_2}\Im (z)$ such that with
 $$S'_{1}=\left\{z:\,  a'<\Im(z)<\Im(c_2)\,\right\},$$
 there exists a conformal map
 \begin{equation}\label{phi21}
 \phi_{2,1}: S'_1\rightarrow S_1
 \end{equation}
 fixing three boundary points $c_2$ and $\pm\infty$ and satisfying $\phi_{2,1}(z+1)=\phi_{2,1}(z)+1$ for any $z\in \overline{S'_1}$.
 \end{lemma}

\begin{proof}
Let 
$$\tilde{S}_1=\left\{z:\, a<\Im(z)<\Im(c_2)\,\right\}.$$
Then by the Riemann mapping theorem there exists a conformal map
$$\psi: \tilde{S}_1 \to S_1$$
which fixes three boundary points $c_2$ and $\pm\infty$. Let $A$ be such that $\psi(A)=c_2 +1$ and put $a'=(a+ic_2)/(A-c_2)-ic_2$. Then the map
$$\phi_{2,1}(z)=\psi\left((A-c_2)(z-c_2)+c_2\right)$$
is a conformal map from $S'_1$ onto $S_1$. It also follows that $\phi_{2,1}$ fixes boundary points $c_2,\, c_2+1$ and $\pm\infty$. We show next that $\phi_{2,1}(z+1)=\phi_{2,1}(z)+1$. For this, consider
$$\tilde{\psi}(z)=\phi_{2,1}(z+1)-1.$$
Then one can deduce that $\tilde{\psi}$ also fixes boundary points $c_2$ and $\pm\infty$. By uniqueness, $\tilde{\psi}(z)=\phi_{2,1}(z)$. So $\phi_{2,1}$ has the property as required.
\end{proof}

\begin{figure}[h]
	\centering
	\includegraphics[width=14cm]{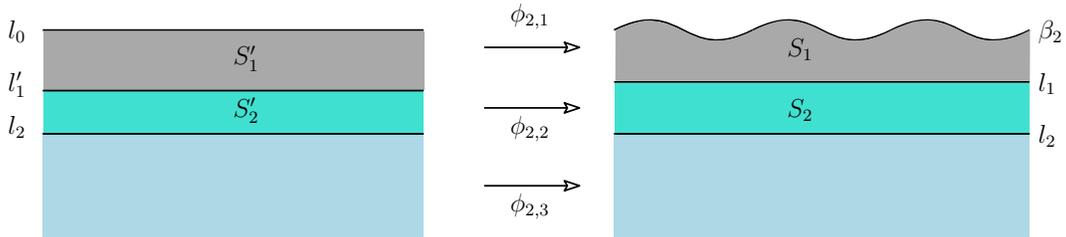}
	\caption{The quasiconformal mapping $\phi_2: H_2\to\tilde{H}_2$ consists of three mappings: $\phi_{2,1}$ is conformal, sending the horizontal strip $S_1'$ to the curvilinear strip $S_1$; the map $\phi_{2,2}$ interpolates between the extension of $\phi_{2,1}$ to $l'_1$ and the identity map on $l_2$; finally $\phi_{2,3}$ is the identity map on the rest of $H_2$.}
	\label{qc}
\end{figure}

Let $a'$ be as found in the above Lemma \ref{pecon}, we can choose a constant $b<\min\{a, a'\}$, and define
$$l_0=\left\{\,z:\,\Im(z)=\Im(c_2)\, \right\},$$
$$l_2=\left\{\,z\in {H}_2:\,\Im z=b\, \right\},$$
$$l'_1=\left\{\,z\in {H}_2:\,\Im z=a'\, \right\}.$$
The strip between $l'_1$ and $l_2$ is denoted by $S'_2$ and the strip between $l_1$ and $l_2$ is denoted by $S_2$. See Figure \ref{qc}. Our next step is the construction of a quasiconformal map 
\begin{equation}\label{phi22}
\phi_{2,2}: S'_2\to S_2.
\end{equation}
For this purpose, we denoted by $\chi_2$ the boundary extension of the conformal map $\phi_{2,1}$ to $l'_1$. It then follows from Lemma \ref{pecon} that $\chi_2: l'_1\to l_1$ is a $\mathcal{C}^1$-diffeomorphism and moreover $\chi_2(z+1)=\chi_2(z)+1$ for any $z\in l'_1$. Then our expected quasiconformal mapping $\phi_{2,2}$ will be constructed as the linear interpolation between $\chi_2$ on $l'_1$ and the identity map on $l_2$.  To define this map, it suffices to define a quasiconformal mapping, with $T'=S'_2-ib$ and $T=S_2-ib$,
$$L: T'\to T,$$
which are linear interpolation between the identity map, denoted by $\chi_1$, on the real axis $\mathbb{R}$ and another map $\chi_2(z+ib)-ib$ on the boundary $l'_1-ib$ of $T'$. If we define $\tilde{\chi}_1(x)=\chi_1(x)$ and $\tilde{\chi}_2(x)=\chi_2(x+ia')-ia$, then both of them are increasing $\mathcal{C}^1$-diffeomorphisms of the real axis. So we can define
$$L(x+iy)=\left(1-\frac{y}{a'-b} \right)\tilde{\chi}_1(x)+\frac{y}{a'-b}\tilde{\chi}_2(x)+i\frac{a-b}{a'-b}y.$$
To show that $L$ is quasiconformal, it is sufficient to check the Jacobian of $L$ is non-zero. For this map, we see by simple computation that its Jacobian is equal to
$$\frac{a-b}{a'-b}\left(\left(1-\frac{y}{a'-b} \right)\tilde{\chi}'_1(x)+\frac{y}{a'-b}\tilde{\chi}'_2(x)\right),$$
which is strictly bigger than zero since both $\tilde{\chi}_1$ and $\tilde{\chi}_2$ are increasing $\mathcal{C}^1$-diffeomorphisms. So, $L$ is a $\mathcal{C}^1$-diffeomorphism, in particular, a quasiconformal map. Denote by $K_2$ the quasiconformal constant of $L$. Then we can put
$$\phi_{2,2}(z)=L(z-ib)+ib,$$
which is a $\mathcal{C}^1$-diffeomorphism, and thus in particular, a quasiconformal map. The quasiconformal constant of $\phi_{2,2}$ is $K_2$.

We also define
\begin{equation}\label{phi23}
\phi_{2,3}: H_2 \setminus \overline{S_1'\cup S'_2}\rightarrow \tilde{H}_2 \setminus \overline{S_1\cup S_2}
\end{equation}
to be the identity map.

With the above constructions given in \eqref{phi21},  \eqref{phi22} and \eqref{phi23}, we can define a map $\phi_2: H_2 \to \tilde{H}_2$ as follows:
\begin{align*}
\phi_2(z)=
\begin{cases}
\,\phi_{2,1}(z) &\mbox{if}~\,z\in S_{1}',\\[0.3em]
\,\phi_{2,2}(z) &\mbox{if}~\,z\in S_{2}',\\[0.3em]
\,\phi_{2,3}(z) &\mbox{elsewhere.}
\end{cases}
\end{align*}
By construction, $\phi_2$ is $K_2$-quasiconformal. The above construction actually shows that $\phi_2$ is a $\mathcal{C}^1$-diffeomorphism. See Figure \ref{qc}. 

\medskip
Clearly by using the same argument for the construction of $\phi_2$ as above we can construct a map which sending $H_1$ to $\tilde{H}_1$ quasiconformally. Let
$$h_1: H_1\to\tilde{H}_1$$
be the obtained ($\mathcal{C}^1$-diffeomorphic) map, whose quasiconformal constant is denoted by $K_1$. The problem now is that we want to have the property \eqref{aim}, which is not necessarily true because $\wp_1(h_1(x+c_1))$ may not coincide with $\wp_2(\phi_2(x+c_2))$ on the real axis, even though they belong to the same curve $\gamma$ defined earlier (see the proof of the Proposition \ref{prop1}). This means that the two maps $\wp_1\circ h_1(z+c_1)$ and $\wp_2\circ\phi_2(z+c_2)$ cannot extend continuously across each other. To solve this, we change $\phi_2$ a little further (it suffices to change one of them). More specifically, we define a "correction" function on the line $l_0$ in the following way:
\begin{align}
k:\, l_0\,&\rightarrow\, l_0\\
z&\mapsto \left(\phi_{2}^{-1} \circ\wp_{2}^{-1} \circ \wp_1 \circ  h_1 \right) (z-c_2+c_1).
\end{align}
The function is not well defined if one does not fix particular inverse branches of $\wp_2$. However, this is not a problem and the inverse branches are chosen according to how the curve $\beta_2$ is mapped onto $\gamma$ (see the proof of the Proposition \ref{prop1}). Moreover, for a fixed $z\in l_0$, its image $h_1(z-c_2+c_1)$ will lie in a segment connecting two points, say $c_1+n$ and $c_1+(n+1)$ on $\beta_1$ for some $n\in\mathbb{Z}$ (recall that $h_1$ is obtained in the same way as $\phi_2$ and thus has the property that $h_1(z+1)=h_{1}(z)+1$). When we consider a preimage of $\wp_1(h_1(z-c_2+c_1))$ under $\wp_2$, we choose the one lying in between $c_2+n$ and $c_2+(n+1)$. The construction of $\beta_2$ ensures that this can be done. It also follows from the construction of $\beta_i$ that the function $\tilde{k}(x)=k(x+c_2)-c_2$ is an increasing $\mathcal{C}^1$-diffeomorphism of the real axis.

\begin{figure}[h]
	\centering
	\includegraphics[width=13cm]{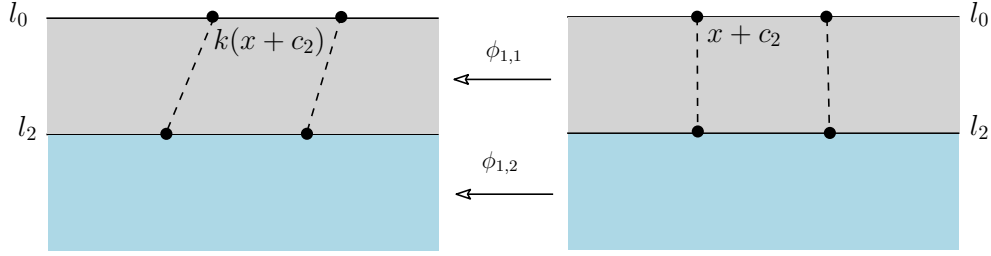}
	\caption{The map $\phi_1$ is a correction map which is used to remove the arising discontinuity. It is constructed by a linear interpolation.}
	\label{qc1}
\end{figure}

Now we consider the linear interpolation between the map $k$ on $l_0$ and the identity map on $ l_2$. This is the same as we have done for $\phi_{2,2}$. By using the map $z\mapsto (z-c_2)$ we can move $l_0$ to the real axis. Then the interpolation between $\tilde{k}$ on the real axis and the identity map on the line $\{z: \Im(z)=b-\Im(c_2)\}$ is given by
$$\tilde{\phi}_{1,1}(x+iy)=\left(1-\frac{y}{b-\Im(c_2)}\right)\tilde{k}(x)+\frac{y}{b-\Im(c_2)}x+iy\quad\text{for~}\,\,b-\Im(c_2)\leq y\leq 0.$$
Then one can see that the Jacobian of $\tilde{\phi}_{1,1}$ is non-zero. Therefore, this gives a $\mathcal{C}^1$-diffeomorphism, and thus by the periodicity, a quasiconformal map, whose quasiconformal constant we denote by $K_3$. Now by setting
$$\phi_{1,1}(z)=\tilde{\phi}_{1,1}(z-c_2)+c_2\quad\text{for~}\,\,b\leq y\leq \Im(c_2),$$
we have defined a $K_3$-quasiconformal map by interpolating between $k$ on $l_0$ and the identity on $l_2$.

Moreover, we put
$$\phi_{1,2}(x+iy)=x+iy.$$
With the above maps $\phi_{1,1}$ and $\phi_{1,2}$ we define a map $\phi_1: H_2\to H_2$ by putting
\begin{align*}
\phi_1(z)=
\begin{cases}
\,\phi_{1,1}(z) &\mbox{if}~\, b \leq\Im z \leq \Im(c_2),\\[0.3em]
\,\phi_{1,2}(z) &\mbox{if}~\,\Im z \leq b.
\end{cases}
\end{align*}
This is clearly a $K_3$-quasiconformal map in $H_2$ (actually a $\mathcal{C}^1$-diffeomorphism). See Figure \ref{qc1}.

Now, replace the map $\phi_2$ by defining
$$h_2=\phi_2 \circ \phi_1.$$
This is a $K_2 K_3$-quasiconformal map sending $H_2$ to $\tilde{H}_2$ (and also a $\mathcal{C}^1$-diffeomorphism). One can check now that the functions $\wp_1 \circ h_1(x+c_1)$ and $\wp_2 \circ h_2(x+c_2)$ agree on the real axis (so we have the property \eqref{aim}). The construction is thereby finished in the sense that we have obtained a quasi-meromorphic function (see Figure \ref{idea})
\begin{equation}\label{newdg}
G(z)=
\left\{ \begin{array}{rl}
\,\wp_1\circ h_1(z+c_1), &\mbox{if}~\, z \in \h^+,\\[0.3em]
\,\wp_2\circ h_2(z+c_2), &\mbox{if}~\,z \in \h^-.
\end{array} \right.
\end{equation}
To recover a meromorphic function one uses the measurable Riemann mapping theorem: There exist a meromorphic function $f$ and a quasiconformal mapping $\phi$ of the plane such that 
\begin{equation}\label{gfphi}
G=f\circ\phi.
\end{equation}

The meromorphic function $f$ belongs to the class $\s$, since we have used (at most) two different Weierstra{\ss} elliptic functions. More precisely, $f$ has at most $7$ critical values in $\hc$ and no asymptotic values. To derive asymptotic behaviours of $f$, we use Lemma \ref{logareaes}. Note that the supporting set of $\phi$ is a horizontal strip and thus has finite logarithmic area. By Lemma \ref{logareaes}, we have
\begin{equation}\label{ascon}
\phi(z)=z+ o(z)
\end{equation}
as $z\to\infty$. This gives us that, for large $r$,
$$n(r,f)\sim n(r,G)\sim 2\,\left(\frac{1}{2}\frac{\pi r^2}{\Im\tau_1}+ \frac{1}{2}\frac{ \pi r^2}{\Im\tau_2}\right)=\pi\left(\frac{1}{\Im\tau_1}+\frac{1}{\Im\tau_2}\right)\,r^2.$$ 
Therefore, it follows from Lemma \ref{tlemmma} that
$$\rho(f)=2.$$
We also claim that
\begin{lemma}
All poles of $f$ are double poles.
\end{lemma}
To prove this, it suffices to check poles of the map $G$, which are poles of two Weierstra{\ss} elliptic functions lying in certain upper or lower half-planes. Therefore, all poles of $G$ are double poles. We also note that $G$ has no poles on the real axis by the choice of $\gamma$ and the constructions of $\beta_i$ shown in the proof of the Proposition \ref{prop1}.

To sum up, we have the following result.
\begin{theorem}\label{sumup1}
There exists uncountably many meromorphic functions $f$ of order $2$ in the class $\s$ satisfying the following properties:
\begin{itemize}
\item[$(1)$] $f$ has no asymptotic values and at most $7$ critical values;
\item[$(2)$] all poles of $f$ have multiplicity $2$.
\end{itemize}
\end{theorem}

Note that the uncountability  in the above theorem follows from the fact that we have uncountably many choices of pairs of distinct Weierstra{\ss} elliptic functions in our construction.

\medskip
\noindent{\emph{Equivalence.}} Here we prove that if one chooses different pairs of Weierstra{\ss} elliptic functions, then the obtained Speiser functions are quasiconformally equivalent. Let $\tau$ be such that $\Im(\tau)>0$. Let also $\tau_1, \tau_2$ be distinct such that $\Im(\tau_i)>0$. Moreover, we also assume that $\tau_i\neq\tau$. Let $\wp$ be the Weierstra{\ss} elliptic function with periods $1$ and $\tau$, and $\wp_i$ the Weierstra{\ss} elliptic functions with periods $1$ and $\tau_i$. With the method of construction in this section, one can obtain two Speiser functions $f_i$, where $f_i$ is the result by gluing $\wp$ and $\wp_i$. Here we show that

\begin{theorem}\label{equd}
$f_1$ is quasiconformally equivalent to $f_2$.
\end{theorem}

\begin{proof}
To prove the theorem, it is sufficient to show that they are topologically equivalent, by \cite[Proposition 2.3 (d)]{epstein5}. In other words, we need to show that there exist two homeomorphisms $\varphi_1,\,\psi_1: \c\to\c$ such that $\varphi_1\circ f_1=f_2\circ\psi_1$. Moreover,  by construction, each $f_i$ can be represented in the form of \eqref{gfphi}. Assume that $f_i=G_i\circ\phi_i$ with $G_i$ quasimeromorphic and $\phi_i$ quasiconformal. Therefore, to prove the theorem it suffices to prove that $G_1$ and $G_2$ are topologically equivalent. This is in some sense similar to the proof of Observation given at the beginning of this section. The essential ingredient of the proof is the constructions of certain graphs in the plane which play the role of the graph obtained by connecting lattice points for Weierstra{\ss} elliptic functions.

Now let $v_1,\dots, v_6$ be the finite critical values of $G_1$, and $w_1, \dots, w_6$ the finite critical values of $G_2$. Both of $G_1$ and $G_2$ also have $\infty$ as a critical value. Choose two closed Jordan curves $\Gamma_1$ and $\Gamma_2$ on $\hc$ such that $\Gamma_i$ passes through $\infty$ and critical values of $G_i$ in the same order. Now each $\Gamma_i$ can be viewed a graph in the obvious way, whose vertices are critical values of $G_i$ and whose edges are the parts of $\Gamma_i$ connecting vertices. Each $\Gamma_i$ decomposes the sphere into two Jordan domains $A_i$ and $B_i$. Put $\Lambda_i=G_{i}^{-1}(\Gamma_i)$. Then each $\Lambda_i$ is a graph embedded in the plane whose vertices are preimages of vertices of $\Gamma_i$ (i.e., preimages of critical values of $G_i$) and whose edges are preimages of edges of $\Gamma_i$. Moreover, each face of $\Lambda_i$ is mapped homeomorphically to either $A_i$ or $B_i$ by $G_i$. 

Let $\varphi: \Gamma_1\to\Gamma_2$ be a homeomorphism fixing $\infty$ and sending $v_i$ to $w_i$ respectively. We then would like to define a map between graphs $\Lambda_i$ using $\varphi_i$ and $G_i$. To achieve this, we need to give a labeling on faces of $\Lambda_i$ which helps us to locate points mapped to each other by the expected graph map on $\Lambda_i$. The construction of $G_i$ implies that the origin is a regular point and lies either in a face of $G_i$ or on the edge of $\Lambda_i$. In the former case, we label $(0,0)$ for the unique face containing the origin for each $G_i$ and assume that this face is mapped to $A_i$. In the later case, we label $(0,0)$ the unique face of $\Lambda_i$ which contains the origin on the boundary and is mapped to $A_i$. By construction of the $G_i$, we know that $\Lambda_i$ induces a tiling of the plane, such that each face is a polygon of $7$ vertices on the boundary. However, although each face has $7$ vertices, only $4$ of them are critical points. These critical points meet exactly four faces. Every other vertex meets two faces. Hence we have a tiling of the plane of quadrilaterals, by ignoring the vertices which are not critical points. (One can think of these quadrilaterals as deformations of parallelograms generated by half-periods of Weierstra{\ss} elliptic functions.) Now the face $(0,0)$ labeled above gives a natural labeling for all the rest of faces by moving along the horizontal and $\tau$ and $\tau_i$ directions: there is only one face which shares a common boundary with the face $(0,0)$ when moving along the positive (respectively negative) real direction and we denote this face by $(1,0)$ (resp. $(-1,0)$); there is also only one face which shares a common boundary with the face $(0,0)$ when moving along $\tau$-direction (resp. $\tau_i$-direction) and we denote this face by $(0,1)$ (resp. $(0,-1)$). We then continue this procedure and in this way every face has a unique labeling. Now we pick up a point $z\in\Lambda_1$. Then $z$ is uniquely determined by the labels of faces which have $z$ on their boundaries. So, $\varphi\circ G_1(z)$ is a point on $\Gamma_2$, which has infinitely many preimages on $\Lambda_2$. However, with the above labeling induced by the point $z$, there is a unique point $w$ in $\Lambda_2$, which has the same corresponding labels for faces adjacent to $w$, and which is mapped to $\varphi\circ G_1(z)$ by $G_2$. So we have just defined a homeomorphism
\begin{equation}\label{commu}
	\psi: \Lambda_1 \to \Lambda_2
\end{equation}
which sends $z$ to $w$.

 Now we extend $\varphi$ to $A_1$ and $B_1$ homeomorphically and thus $\psi$ can be extended homeomorphically to the faces of $\Lambda_1$ due to the fact that every face can be mapped homeomorphically to one of $A_1$ and $B_1$. In this way, the above \eqref{commu} is extended to the whole plane. We thus have
$$G_2\circ\psi=\varphi\circ G_1,$$
as required.
\end{proof}

\medskip
For the estimate of the Hausdorff dimension of escaping set for $f$, we will also need to have an understanding of local behaviours of $f$ near its poles. First let $z_0$ be a pole of $G$. Then by \eqref{newdg}, $h_i(z_0+c_i)$ is a pole of $\wp_i$ for some $i\in\{1,2\}$. Without loss of generality, we assume that $i=1$. Put $\zeta=h_{1}(z+c_1)$ and $\zeta_0=h_{1}(z_0+c_1)$. So $\zeta_0$ is a pole of $\wp_1$. So there exists a constant $C'$ such that
$$G(z)\sim\left(\frac{C'}{\zeta-\zeta_0}\right)^2\,\quad\text{as}\quad z\to z_0.$$
Note that, by our construction $h_1$ is $\mathcal{C}^1$-diffeomorphic and periodic. So one sees that $\zeta-\zeta_0\sim C''(z-z_0)$ for some universal constant $C''$ as $z\to z_0$. With this we deduce that
$$G(z)\sim\left(\frac{C}{z-z_0}\right)^2\,\quad\text{as}\quad z\to z_0,$$
where $C$ is a constant depending on $C'$ and $C''$.

Put $w=\phi(z)$ and $w_0=\phi(z_0)$. Since $w_0$ is a double pole of $f$, we may assume that
$$f(w)\sim \left(\frac{b(w)}{w-w_0}\right)^2\,\quad\text{as}\quad w\to w_0.$$
Here $b(w)$ is a holomorphic function near $w_0$ and $b(w_0)\neq 0$. The above two estimates, together with the relation \eqref{gfphi}, imply that
$$b(w)\sim C\,\frac{w-w_0}{z-z_0}\,\quad\text{as}\quad z\to z_0.$$
Recall that $\phi$ is conformal at $\infty$, so one also has $w\to w_0$ as $z\to z_0$. Now with \eqref{ascon} we see that
$$b(w_0)\sim C\,\lim_{z\to z_0}\frac{w-w_0}{z-z_0}=C.$$
So near the pole $w_0$, one has
$$f(w)=\left(\frac{B}{w-w_0}\right)^2\,\quad\text{as}\quad w\to w_0,$$
where $B$ is certain constant depending on $C$.

\subsection{Speiser functions with orders in $(0,2)$}\label{sec3.2}

This part is devoted to constructing Speiser meromorphic functions with orders in $(0,2)$. Basically we follow the construction above and also have to make substantial changes on certain parts. The main idea here, compared with the surgery in the previous section, is then to glue instead two functions of the form $\wp_i \circ h_i (z^{\eta})$ for $i=1,2$, along the real axis with $h_i$ quasiconformal and two carefully chosen Weierstra{\ss} elliptic functions $\wp_i$ for $\eta\in (0,1)$. For a sketch of the idea for the construction here, see Figure \ref{idea2}.

\begin{figure}[h]
	\centering
	\includegraphics[width=15cm]{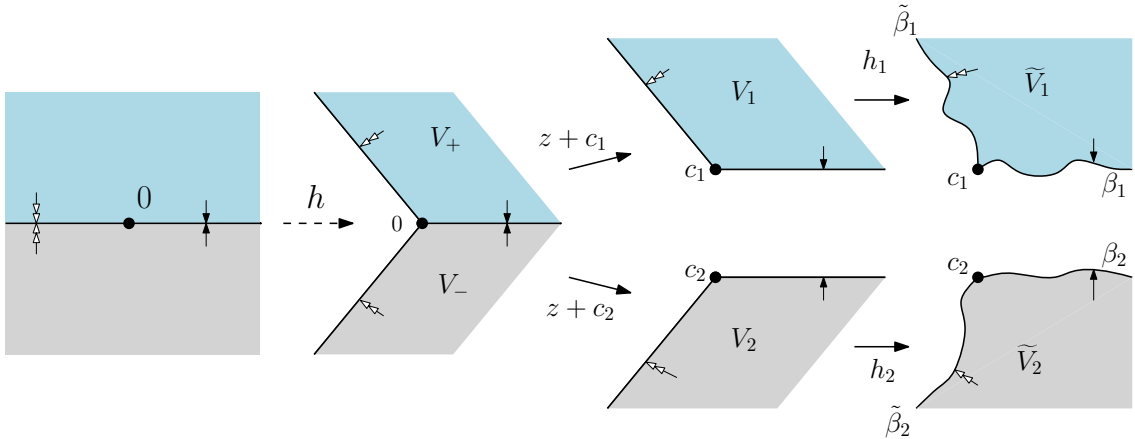}
	\caption{A sketch of the idea of construction for Speiser functions with orders in $(0,2)$. The map $h$ is essentially used to change the orders while quasiconformal mappings $h_1$ and $h_2$ are used to remove the discontinuity along the real axis.}
	\label{idea2}
\end{figure}

To be more specific, we will prove the following result.

\begin{theorem}\label{leq2}
	For any given $\rho\in(0,2)$, there exists a Speiser meromorphic function $f$ of order $\rho$ with at most $7$ critical values and no asymptotic values. Moreover, all poles of $f$ have multiplicity $2$.
\end{theorem}

Let $\rho\in(0,2)$ be given. Put $\alpha=\rho\pi\in (0,2\pi)$. Let $\wp$ be a Weierstra{\ss} elliptic function with two periods $1$ and $\tau$, where $\tau$ is non-real such that $\Im \tau>0$ and
$$\tan\frac{\alpha}{2}=\frac{\Im \tau}{\re\tau}.$$
Put
\begin{align}
V_{+}&=\left\{\,z=re^{i\theta}:\, r>0,\, 0<\theta<\frac{\alpha}{2} \,\right\},\\
V_{-}&=\left\{\,z=re^{i\theta}:\, r>0,\, -\frac{\alpha}{2}<\theta<0 \,\right\}
\end{align}
and
$$V=\left\{z=re^{i\theta}:\, r>0,\, |\theta|<\frac{\alpha}{2}\, \right\}.$$
Then the function
\begin{equation}\label{maph}
h: \,\c\setminus\mathbb{R}^{-}\to V,\,\,~\,z\mapsto z^{\frac{\alpha}{2\pi}}
\end{equation}
is a conformal map, where we use the principal branch of the logarithm. At the discontinuity on the negative real line we will consider two slightly different versions of the principal branch as follows; namely, for $x < 0$ we define  
\begin{align}
h_+(x) &= x^{\frac{\alpha}{2\pi}}\quad  \text{ where~ $\arg(x) = \pi$},\\
h_-(x) &= x^{\frac{\alpha}{2\pi}}\quad  \text{ where~ $\arg(x) = -\pi$}.
\end{align}
For all other $z \neq 0$ we let $h_+$ and $h_-$ be defined as the usual principal branch. 

We now consider two $\wp$-functions: $\wp_1$ with periods $1$ and $\tau_1 = \tau$  and $\wp_2$ with periods $1$ and $\tau_2=\overline{\tau}$. Then we want to glue $\wp_1\circ h$ with $\wp_2\circ h$ along the positive real axis using the methods described in Section \ref{qc-gluing}. Again, the arising discontinuity along the real axis presents certain obstructions. To overcome this, we use similar idea as in the previous section. We use notations as given in \eqref{symbol1} and \eqref{symbol2} for the critical values of $\wp_i$ here.

Our starting point is a result analogous to Proposition \ref{prop1}. 

\begin{proposition}\label{prop2}
There exist two analytic closed curves $\gamma$ and $\tilde{\gamma}$ and two points $c_i$ in the plane with the following properties:
\begin{itemize}
\item[$(1)$] $\gamma$ separates the critical values $v_3,\,v_4,\, w_3,\,w_4$ from all other critical values of $\wp_i$ for $i=1,2$. For each $i$ there exists an unbounded analytic curve $\beta_i$ starting from $c_i$ such that $\wp_i(\beta_i)=\gamma$. Moreover, $\beta_i$ is periodic with period $1$ (i.e., $z\in\beta_i$ implies $z+1\in\beta_i$).
\item[$(2)$] $\tilde{\gamma}$ separates the critical values $v_2,\,v_3,\, w_2,\,w_3$ from all other critical values of $\wp_i$ for $i=1,2$. For each $i$ there exists an unbounded analytic curve $\tilde{\beta}_i$ starting from $c_i$ such that $\wp_i(\tilde{\beta}_i)=\tilde{\gamma}$. Moreover, $\tilde{\beta}_i$ is periodic with period $\tau_i$.
\item[$(3)$] $\gamma$ intersects with $\tilde{\gamma}$ at exactly two points.
\end{itemize}
\end{proposition}

\begin{proof}
As in the proof of Proposition \ref{prop1}, we choose a Jordan curve $\Gamma$ passing through all critical values of $\wp_i$ with the same conditions there. Then by choosing an analytic curve $\gamma$  separating $v_3, v_4$ and $w_3, w_4$ from all other critical values of $\wp_i$ and intersecting with $\Gamma$ exactly twice, we can obtain two periodic analytic curves $\widehat{\beta}_i$ along the horizontal direction and such that $\wp_i\left(\widehat{\beta}_i\right)=\gamma$. 

Now we choose another analytic curve $\tilde{\gamma}$ separating $v_2,\,v_3,\, w_2,\,w_3$ from all other critical values of $\wp_i$ for $i=1,2$ and intersecting with $\Gamma$ exactly twice. Moreover, we assume that $\tilde{\gamma}\cap\gamma$ also consists of two points. This can be done since both of them are analytic curves. By using the construction in Proposition \ref{prop1}, we can obtain two periodic analytic curves $\widehat{\tilde{\beta}}_i$ which are $\tau_i-$periodic; in other words, $z\in\widehat{\tilde{\beta}}_i$ implies that $z+\tau_i\in \widehat{\tilde{\beta}}_i$.

Since $\tilde{\gamma}$ intersects with $\gamma$ consists of two points, we see that $\widehat{\beta}_i\cap\widehat{\tilde{\beta}}_i$ consists of exactly one point, denoted by $c_i$. Now we define $\beta_i$ to be the part of $\widehat{\beta}_i$ starting from $c_i$ and tending to $\infty$ along the direction of the positive real axis. Similarly we let $\tilde{\beta}_i$ be the part of $\widehat{\tilde{\beta}}_i$ along the direction of $\tau_i$ starting from $c_i$. See Figure \ref{idea2}. This finishes our construction of desired curves.
\end{proof}

Put
\begin{equation}\label{desh}
V_{1}=V_{+}+c_1\quad\text{and}\quad V_{2}=V_{-}+c_2.
\end{equation}
We also define $\widetilde{V}_i$ to be the domain bounding by the curves $\beta_i$ and $\tilde{\beta}_i$. Then our main focus will be the construction of the quasiconformal mappings
$$h_{i}: V_{i}\to\widetilde{V}_i$$
such that
\begin{equation}\label{finapur}
\wp_{1}\left(h_{1}(h_+(x)+c_1) \right)=\wp_{2}\left(h_{2}(h_-(x)+c_2) \right)\quad\text{whenever}\quad x\in\mathbb{R}.
\end{equation}
As one might notice, the $h_i$'s here play the same role as $h_i$'s in the previous section except that there we are requiring $h_i$ to be a map defined in sectors instead of certain half-planes.

\medskip
Now we start from the construction of $h_i$ in part of $V_i$. First we note that $\wp_1$ and $\wp_2$ satisfy conditions that are required for the gluing in the previous section. So we can glue $\wp_1$ and $\wp_2$ along the \emph{whole} real axis using exactly the same method there. More precisely, let $\widehat{\beta}_i$ be as constructed in the proof of Proposition \ref{prop2} (recall that $\beta_i$ are part of $\widehat{\beta}_i$). Let $H_1=\{z: \Im(z)>\Im(c_1)\}$ and $H_2=\{z: \Im(z)<\Im(c_2)\}$, and also $\tilde{H}_1$ the domain bounded below by $\widehat{\beta}_1$ and $\tilde{H}_2$ be the domain bounded above by $\widehat{\beta}_2$. Then we can obtain a quasiconformal mapping
$$\hat{h}_1:\, H_1\to\tilde{H}_1$$
which is quasiconformal in the strip
$$\left\{z: \Im(c_1)<\Im(z)< b_1 \right\}$$
for some $b_1$ and is identity elsewhere in $H_1$, and another quasiconformal mapping
$$\hat{h}_2:\, H_2\to\tilde{H}_2$$
which is quasiconformal in the strip
$$\left\{z: b_2<\Im(z)<\Im(c_2)\right\}$$
for some $b_2$ and is identity elsewhere in $H_2$. Moreover, we have
\begin{equation}
\wp_1\left( \hat{h}_1\left(x+c_1\right)\right)=\wp_2\left( \hat{h}_2\left(x+c_2\right)\right)\quad\text{for}\quad x\in\mathbb{R}.
\end{equation}
Put
\[
S_1 = \left\{ z\in V_{1}: \Im(c_1) < \Im(z) < b_1\, \right\}
\]
and
\[
S_2 = \left\{ z\in V_{2}: b_2<\Im(z)<\Im(c_2)\, \right\}.
\]
We define $h_1$ and $h_2$ on $S_1$ and $S_2$ respectively as  the restriction of $\hat{h}_1$ and $\hat{h}_2$ on $S_1$ and $S_2$. With this we see immediately that \eqref{finapur} holds for $x>0$. So the remaining work is to make sure that \eqref{finapur} is also true for $x<0$. This will be done using similar arguments as before. We will do construction again in domains $V_1$ and $V_2$. Instead of doing quasiconformal surgery along horizontal direction, this time we work along $\tau_1$ and respectively $\tau_2$ directions.

\medskip
Recall the analytic curves $\widehat{\tilde{\beta}}_i$ constructed in the proof of Proposition \ref{prop2} (and $\tilde{\beta}_i$ is a "half" of $\widehat{\tilde{\beta}}_i$). We put
$$l_+=\left\{ z +c_1: \arg(z) = \arg(\tau_1) \,\right\}.$$
\begin{figure}[h]
	\centering
	\includegraphics[width=15cm]{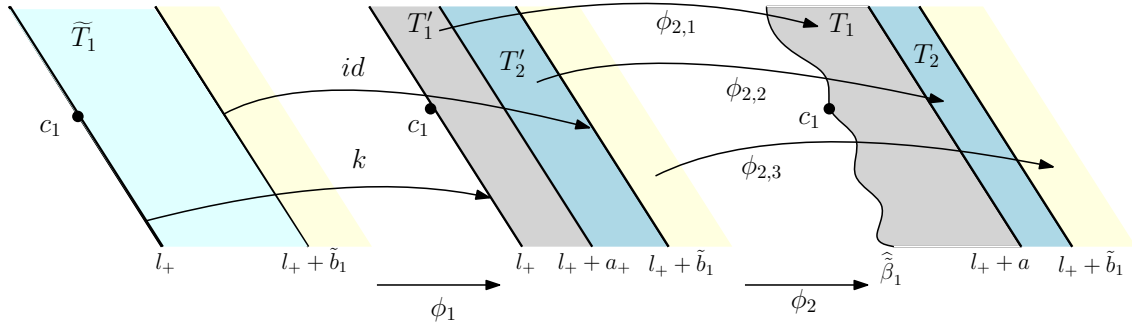}
	\caption{The map $\tilde{h}_1: R_1\to \tilde{R}_1$ is the composition of two quasiconformal mappings $\phi_2$ and $\phi_1$. Roughly speaking, one can view $\phi_2$ as a "straightening" map while $\phi_1$ is a "correction" map which fixes the difference between $\phi_2$ and $\tilde{h}_2$.}
	\label{xieperiod}
\end{figure}
Now we choose $a$ such that the line $l_{+}+a$ lies on the right of $\widehat{\tilde{\beta}}_1$. It follows, using the same argument as in the proof of Lemma \ref{pecon}, that there exists a real number $a_{+}$ such that the strip $T'_{1}$ between $l_{+}$ and $l_{+}+a_{+}$ is mapped conformally  onto the strip $T_{1}$ between $\widehat{\tilde{\beta}}_1$ and $l_{+}+a$. This map is denoted by $\phi_{2,1}$. Now we can choose a number $\tilde{b}_1$ such that the line $l_{+}+\tilde{b}_1$ is parallel to $l_{+}$ and lies to the right of both $l_{+}+a$ and $l_{+}+a_{+}$. Denote by $T'_2$ the strip between $l_{+}+a_{+}$ and $l_{+}+\tilde{b}_1$ and by $T_2$ the strip between $l_{+}+a$ and $l_{+}+\tilde{b}_1$. Then we construct a quasiconformal map $\phi_{2,2}$ between $T'_2$ and $T_2$ (as we did for the map $\phi_{2,2}$ in Section \ref{qc-gluing}) by interpolating between the extension of the above map $\phi_{2,1}$ to $l_{+}+a_+$ and the identity map on $l_{+}+\tilde{b}_1$. We omit details for this as it is the same as before. We also consider the identity map $\phi_{2,3}$ on the domain which is to the right of $l_{+}+\tilde{b}_1$. Let us denote by $R_{1}$ the half-plane to the right of the line $l_{+}$ and $\tilde{R}_1$ the curved half-plane lying to the right of $\widehat{\tilde{\beta}}_1$. In this way, we have obtained a quasiconformal map 
$$\phi_{2}: R_1 \to \tilde{R}_1$$
by defining
\[
\phi_2(z)=
\left\{ \begin{array}{rl}
\phi_{2,1}(z), &\quad\text{if} ~z \in T'_1, \\
\phi_{2,2}(z), &\quad\text{if} ~z \in T'_2, \\
\phi_{2,3}(z), &\quad\text{elsewhere.}
\end{array} \right.
\]

With
$$l_-=\left\{ z+c_2 : \arg(z) = \arg(\tau_2) \,\right\},$$
and $R_2$ the half-plane to the right of $l_{-}$ and $\tilde{R}_2$ the curved half-plane to the right of $\widehat{\tilde{\beta}}_2$, we can use the above construction to obtain a quasiconformal map between $R_2$ and $\tilde{R}_2$, which is quasiconformal only in a strip but is identity elsewhere. This map is denoted by $\tilde{h}_2$ (which plays the role of $h_1$ in the last section). Now if we choose two points $x_1$ and $x_2$ respectively on $l_{+}-c_1$ and $l_{-}-c_2$ with the same modulus (i.e., their arguments are $\arg(\tau_1)$ and $\arg(\tau_2)$ respectively), it is still possible that $\wp_{1}(\phi_2(x_1+c_1))\neq \wp_{2}(\tilde{h}_2(x_2+c_2))$. This means that if we choose a point $x\in\mathbb{R}^{-}$, then
$$\wp_1\circ\phi_2\left(h_+(x)+c_1\right)\neq \wp_{2}\circ\tilde{h}_2(h_-(x)+c_2).$$
This, in turn, means that there are discontinuities along the negative real axis. So, as in the previous section, we need to change $\phi_2$ a little further so as to solve this. More precisely, we first define a function 
\begin{align}
k: l_+&\rightarrow l_+\\
z&\mapsto \left(\phi_{2}^{-1} \circ\wp_{1}^{-1} \circ \wp_2 \circ  \tilde{h}_2 \right) \left(\overline{z-c_1}+c_2\right).
\end{align}
This function is essentially obtained in the same manner as the function $k$ in the last section. We omit details but conclude that we interpolate between the function $k$ on $l_{+}$ and the identity map on $l_{+}+\tilde{b}_1$ to obtain a quasiconformal map $\phi_{1,1}$ on the strip, denoted by $\widetilde{T}_1$, between $l_{+}$ and $l_{+}+\tilde{b}_1$. Then a quasiconformal map $\phi_1: R_1\to R_1$ is defined by putting $\phi_1=\phi_
{1,1}$ in the strip between $\widetilde{T}_1$ and $\phi_1(z)=z$ elsewhere. Now put
$$\tilde{h}_1=\phi_2\circ\phi_1.$$
This is a map sending $R_1$ to $\tilde{R}_1$ quasiconformally and has the property that
$$\wp_1\left(\tilde{h}_1(h_+(x)+c_1)\right)=\wp_2\left(\tilde{h}_2(h_-(x)+c_2)\right)\quad\text{for}\quad x\in\,\mathbb{R}^-.$$

\medskip
We have defined two quasiconformal mappings $\hat{h}_i$ and $\tilde{h}_i$ on each of $V_i$. Without loss of generality, we take $V_1$ as an example and consider the restrictions of $\hat{h}_1$ and $\tilde{h}_1$ in $V_1$. The case for $V_2$ goes in the same way. Then the map $\hat{h}_1$ is quasiconformal in the horizontal half-strip $S_1$ and identity on $V_1 \setminus S_1$, while $\tilde{h}_1$ is quasiconformal in the strip $\widetilde{T}_1 \cap V_1$ and is the identity on $V_1 \setminus \widetilde{T}_1$. So they share a sector domain $V'_{1}$ on which both of $\hat{h}_i$ and $\tilde{h}_i$ are the identity map. See Figure \ref{paralcon} for an illustration.
\begin{figure}[h]
	\centering
	\includegraphics[width=12cm]{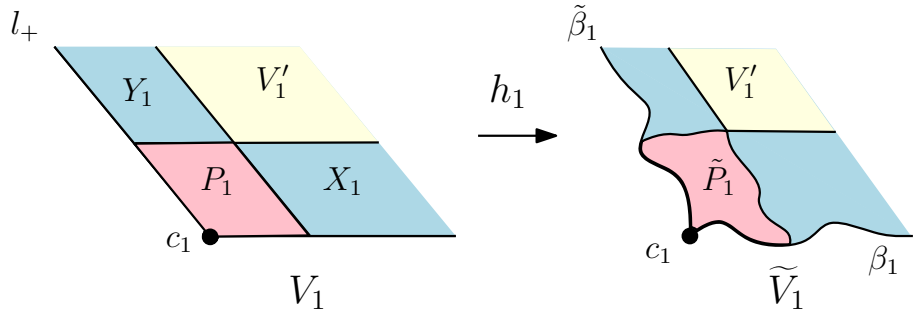}
	\caption{The quasiconformal homeomorphism $h_1: V_1\to \tilde{V}_1$ is constructed piecewise: it is $\hat{h}_1$ on $X_1$, $\tilde{h}_1$ on $Y_1$, $\breve{h}_{1}$ on $P_1$ and identity elsewhere.}
	\label{paralcon}
\end{figure}
One can also notice that there is a parallelogram with $c_1$ as a vertex, on which $\hat{h}_1$ and $\tilde{h}_1$ may not coincide. We will need to redefine a map on this parallelogram. Denote by $P_1$ the parallelogram as shown in Figure \ref{paralcon}. The edges of $P_1$ are denoted by $\ell_i$ for $i=1,\dots,4$. See Figure \ref{fig42}. We also put
$$\tilde{\ell}_i=\hat{h}_1(\ell_i)\quad\text{for}\quad i=1,2,$$
and
$$\tilde{\ell}_i=\tilde{h}_1(\ell_i)\quad\text{for}\quad i=3,4.$$
The quadrilateral enclosed by $\tilde{\ell}_i$ is denoted by $\tilde{P}_1$. Therefore, we have defined a boundary map between $P_1$ and $\tilde{P}_1$ by using $\hat{h}_1$ on $\ell_1\cup\ell_2$ and $\tilde{h}_1$ on $\ell_3\cup\ell_4$. These are all $\mathcal{C}^1$-diffeomorphisms, as these are extensions of quasiconformal mappings obtained by linear interpolations between $\mathcal{C}^1$-diffeomorphisms. So, by \cite[Lemma 2.24]{branner3}, we can extend the boundary map to $P_1$ and obtain a quasiconformal map
$$\breve{h}_{1}: P_1\to\tilde{P_1}.$$

\begin{figure}[h] 
	\centering
	\includegraphics[width=10cm]{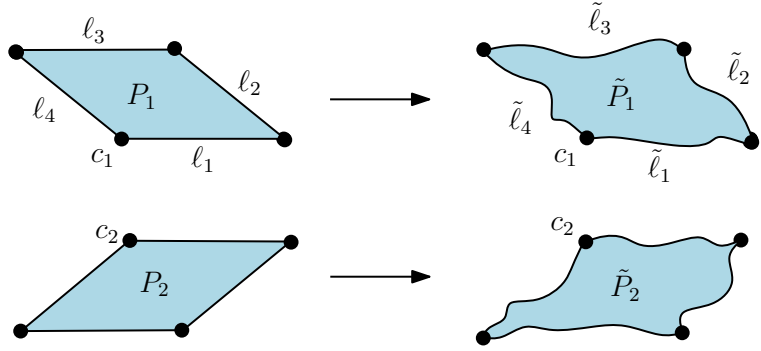}
	\caption{Two quasiconformal homeomorphisms from $\breve{h}_{i}: P_i\to\tilde{P}_i$ are constructed by extending boundary maps $\hat{h}_i$ and $\tilde{h}_i$.}
	\label{fig42}
\end{figure}
Analogously, we can do this for the sector $V_2$ and get a quasiconformal map
$$\breve{h}_{2}: P_2\to\tilde{P}_2,$$
where $P_2$ is defined similarly as $P_1$ and so is $\tilde{P}_2$ (see Figure \ref{fig42}). Put 
$$X_1=\left(V_1\cap S_1\right)\setminus \widetilde{T}_1$$
and
$$Y_1=\left(V_1\cap \widetilde{T}_1\right)\setminus S_1.$$

Now we can define our desired map $h_1$ as follows:
\begin{align*}
h_1(z)=
\begin{cases}
\,\hat{h}_1(z) &\mbox{if}~~z \in X_1,\\
\,\tilde{h}_1(z)&\mbox{if}~~z \in Y_1,\\
\,~z &\mbox{if}~~z \in V'_1,\\
\,\breve{h}_{1}(z)&\mbox{if}~~z\in P_1.
\end{cases}
\end{align*}
If $X_2,\,Y_2,\,V'_2$ and $P_2$ are the corresponding sets for $V_2$, then we define
\begin{align*}
h_2(z)=
\begin{cases}
\,\hat{h}_2(z) &\mbox{if}~~z \in X_2,\\
\,\tilde{h}_2(z)&\mbox{if}~~z \in Y_2,\\
\,~z &\mbox{if}~~z \in V'_2,\\
\,\breve{h}_{2}(z)&\mbox{if}~~z\in P_2.
\end{cases}
\end{align*}

Now we can consider the function
\begin{equation}\label{defG}
G(z)=
\left\{ \begin{array}{rl}
\wp_1 \circ h_1( h(z)+c_1), \quad\text{if} ~\Im(z) \geq 0; \\
\wp_2 \circ h_2( h(z)+c_2), \quad\text{if} ~\Im(z) < 0.
\end{array} \right.
\end{equation}
It follows from our construction that \eqref{finapur} holds. Therefore, the function $G$, defined in the above way, is continuous throughout the whole plane and thus is a quasi-meromorphic function. By the measurable Riemann mapping theorem there exists a meromorphic function $f$ and a quasiconformal map $\phi: \c\to\c$ such that
\begin{equation}\label{Gandf}
G = f \circ \phi.
\end{equation}

Since we have used two Weierstra{\ss} elliptic functions, the singular values of $f$ will be the singular values of the two Weierstra{\ss} $\wp$-functions. Therefore, $f$ will have at most $6$ finite critical values and one critical value at $\infty$.

\begin{lemma}\label{polemu}
All poles of $f$ have multiplicity $2$.
\end{lemma}

\begin{proof}
Since $\phi$ is a homeomorphism, it suffices to prove that all poles of $G$ are double poles. Note also that the map $h$ is conformal. Thus it suffices to check the poles of the Weierstra{\ss} elliptic functions $\wp_1$ and $\wp_2$ lying in $\tilde{V}_1$ and $\tilde{V}_2$. This is clear, since no poles $\wp_i$ are lying on the boundaries of $\tilde{V}_i$ by our choice of $\beta_i$ and $\tilde{\beta}_i$, which are certain preimages of analytic curves $\gamma$ and $\tilde{\gamma}$ not passing through $\infty$ (see the proof of Proposition \ref{prop2}).
\end{proof}

\medskip
\noindent{\emph{Asymptotic behaviours.}} To derive the asymptotic behaviours of the function $f$ near $\infty$ we use  Theorem \ref{twb}, which reduces to check, for our purposes, whether the supporting set of $\phi$ has finite logarithmic area. By our construction, the image of the supporting set $X$ under the conformal map $h$ is a union of two strips. Therefore,
$$\iint_{X\setminus\overline{\ud}}\frac{dxdy}{x^2+y^2}=\frac{4\pi^2}{\alpha^2}\iint_{h(X)\setminus\overline{\ud}}\frac{dxdy}{x^2+y^2}<\infty.$$
So, by Lemma \ref{logareaes} we know that $\phi$ is conformal at $\infty$ and thus may be normalised as
\begin{equation}\label{ac}
\phi(z)= z + o(z)\,\,\text{~as~}\,z\to\infty.
\end{equation}
This is then used to show below that
\begin{proposition}\label{02order}
$$\rho(f)=\frac{\alpha}{\pi}=\rho.$$
\end{proposition}
To see this, choose a closed disk $\overline{D}(0,r)$. Then the number of poles, counting multiplicities, of $f$ contained in this disk can be estimated by using \eqref{ac} as
\begin{equation}\label{howmany}
n(r,f)\sim n(r,G)\sim 2\,\frac{\frac{\alpha}{2\pi}\,\pi\cdot r^{\alpha/\pi}}{\Im\tau}=\frac{\alpha\,r^{\rho}}{\Im\tau}
\end{equation}
for sufficiently large $r$. According to Lemma \ref{tlemmma}, we see that
$$\rho(f)=\limsup_{r\to\infty}\frac{\log n(r,f)}{\log r}=\rho.$$
Since $\rho$ ranges in $(0,2)$, we see that any prescribed order in $(0,2)$ can be achieved.

\begin{remark}
	In certain cases, one can actually obtain functions with fewer critical values. For instance, if one takes $\alpha=\pi$, then the sector $V$ will be the right half plane. Without using two $\wp$-functions one can apply directly one Weierstra{\ss} elliptic function with periods $1$ and $\tau$, where $\tau$ is purely imaginary.
\end{remark}

\medskip
\noindent{\emph{Equivalence.}} Similarly as in the previous section, the functions obtained are equivalent except for certain special case. To be more specific, for $i=1,2$ we let $\rho_i\in (0,2)$ and $\rho_i\neq 1$. Suppose that the obtained functions are $f_i$. Then we have the following result.
\begin{theorem}\label{equd1}
	$f_1$ is quasiconformally equivalent to $f_2$.
\end{theorem}

Since there are uncountably many choices of $\rho_i$, we can have uncountably many quasiconformally equivalent Speiser functions but with different orders in $(0,2)$. The proof of the above theorem is similar as Theorem \ref{equd}, so we omit its proof.

\medskip
\noindent{\emph{Local behaviours near poles}.} We first consider local behaviours of the quasi-meromorphic function $G$ near its poles. Suppose that $z_0$ is a pole of $G$. It follows from our construction of $G$ in \eqref{defG} that $\zeta_0=h_i(z_{0}^{\rho/2}+c_i)$ is a pole of one of two Weierstra{\ss} $\wp$-function $\wp_i$. Assume that $i=1$. Put $\zeta=h_1(z_{0}^{\rho/2}+c_1)$. Use similar arguments as in the last section, we conclude that, there exists some constant $B_1$ not depending on the pole such that
\begin{equation}\label{res1}
G(z)\sim \left(\frac{B_1}{z^{\rho/2}-z_{0}^{\rho/2}}\right)^2\,\,\,\text{~as~}\,\,\,z\to z_0.
\end{equation}
By Lemma \ref{polemu}, $z_0$ is a double pole of $G$. We may thus assume that
\begin{equation}\label{res2}
G(z)\sim \left(\frac{B(z)}{z-z_{0}}\right)^2\,\,\,\text{~as~}\,\,\,z\to z_0,
\end{equation}
where $B(z)$ is holomorphic in a small neighbourhood of $z_0$ and moreover $B(z_0)\neq 0$. By comparing \eqref{res1} and \eqref{res2}, we see that
\begin{equation}\label{qmb}
B(z_0)= B_2\,z_{0}^{1-\rho/2}
\end{equation}
for some constant $B_2$. Since $G=f\circ\phi$, we see that $\phi(z_0)$ is a double pole of $f$. Put $w=\phi(z)$ and $w_0=\phi(z_0)$. By Lemma \ref{polemu}, we may assume that
$$f(w)\sim \left(\frac{b(w)}{w-w_{0}}\right)^2\,\,\,\text{~as~}\,\,\,w\to w_0,$$
where $b(w)$ is holomorphic in a small neighbourhood of $w_0$ and $b(w_0)\neq 0$. Therefore, it follows from the relation $G=f\circ\phi$ and \eqref{res2} and \eqref{qmb} that
$$\frac{b(w)}{w-w_0}\sim\frac{B(z)}{z-z_0}\,~\,~\,\text{as}~\,\,z\to z_0.$$
Note that $w\to w_0$ as $z\to z_0$ by Lemma \ref{logareaes}. So with \eqref{qmb}, we have
$$b(w_0)\sim B(z_0)\lim_{z\to z_0}\frac{w-w_0}{z-z_0}=B_3 z_{0}^{1-\rho/2},$$
where $B_3$ is certain constant. Again, by using Lemma \ref{logareaes}, we can have
\begin{equation}\label{res3}
b(w_0)= B_4\, w_{0}^{1-\rho/2},
\end{equation}
where $B_4$ is some constant.

\subsection{Speiser functions with order in $(2,\infty)$}\label{sec3.3}

The original construction of a meromorphic map with order $\rho \in (2,\infty)$ is here replaced by a simpler argument suggested to us by W. Bergweiler, where we use the result from the previous section. Let now $\rho  > 2$, and put $N = \lfloor \rho \rfloor$ and
$\sigma = \rho/N$. Then $\sigma \in [1,2)$. Hence we can find a meromorphic function $f$ with order $\sigma$ from Theorem \ref{leq2}. Now consider the map
\[
g(z) = f(z^N).
\]
Then this gives us the function with order bigger than $2$. Notice that from our construction in the last section, $0$ is not a pole or critical point of $f$. Thus the function $g$ also has only double poles, which are preimages of poles of $f$ under $z^{N}$. Moreover, one can check that $g$ has one more critical value which is $f(0)$. Thus we have

\begin{theorem}\label{geq2}
	For any given $\rho\in(2,\infty)$, there exists a Speiser meromorphic function $g$ of order $\rho$ which has no asymptotic values and has at most $8$ critical values. All poles are double poles.
\end{theorem}

The local behaviour near poles for $g$ will be similar to the description in the previous section for the case $\rho \in (0,2)$. More precisely, if 
$$g(w)\sim \left(\frac{B(w)}{w-w_{0}}\right)^2\,\,\,\text{~as~}\,\,\,w\to w_0,$$
where $B$ is holomorphic near $w_0$ and non-zero at $w_0$, then 
\begin{equation}\label{res4}
B(w_0)= B_4 \, w_{0}^{1-\rho/2},
\end{equation}
where $B_4$ as before is some constant.

\medskip
\noindent{\emph{Equivalence}.} The existence of uncountably many quasiconformally equivalent meromorphic functions with different orders in this situation can be proved in the same way as in Theorem \ref{equd}. So we omit details here and only state the result as follows.

\begin{theorem}\label{thm42}
There exist uncountably many meromorphic functions in the Speiser class which are mutually quasiconformally equivalent but of different orders in $(2,\infty)$.
\end{theorem}

\section{Hausdorff dimension of escaping sets}\label{eses}

In this section we prove our theorem: every number in $[0,2]$ can be  the Hausdorff dimension of escaping sets of certain Speiser functions.

\subsection{Escaping sets of zero dimension}

There are indeed Speiser meromorphic functions whose escaping sets have zero Hausdorff dimension. In fact, it follows from \cite[Theorem 1.1]{bergweiler2} that any class $\b$ meromorphic function of zero order with bounded multiplicities of poles will have escaping sets of zero Hausdorff dimension. One such example is given as follows: Let $\tau=2\pi i$, consider a lattice defined as
$$\Lambda=\{m+n\tau: m, n\in\mathbb{Z}\}.$$
We denote by $\wp=\wp_{\Lambda}$ the Weierstra{\ss} elliptic function with respect to the above lattice. Put
$$S=\left\{z=x+iy:\,x>0,\,|y|<\pi\right\}.$$
Then we can take a branch of  the inverse of $\cosh$, denoted by  $\varphi$, such that
$$\varphi: \c\setminus (-\infty, 1]\to S$$
is conformal. By setting
$$f(z)=\wp(\varphi(z)),$$
we see that $f$ is meromorphic in $\c\setminus (-\infty, 1]$. To see that this function is actually meromorphic in the whole plane, we need to show that the above $f$ extend continuously across $(-\infty,1]$. Now take any $x\in (-\infty, 1]$. Then the extension of $\varphi$ to $(-\infty,1]$ on both sides will map $x$ respectively to two points $z_{1}$ and $z_2$ on $\partial S$ which are complex conjugate. It then follows from the our choice of periods of $\wp$ that $\wp(z_1)=\wp(z_2)$. Thus we have a function meromorphic in the plane, denoted again by $f$. It follows from the construction that $f$ is a Speiser function which has no asymptotic values and four critical values at the critical values of $\wp$. The order of $f$ can be obtained by considering the counting function of poles, which are located at $\cosh(m)$, where $m\in\mathbb{N}$. By computation, one see that the order $\rho(f)=0$. All poles of $f$ are double poles, except for the one at $z=1$, which is a simple pole. Thus this function satisfies the condition of \cite[Theorem 1.1]{bergweiler2}. Therefore, the Hausdorff dimension of the escaping set of this function is zero.

\subsection{The general case}\label{generalcase}

In this part, we will estimate the Hausdorff dimension of the escaping set of the function constructed in Section \ref{CS}. More precisely, we will show the following.

\begin{theorem}\label{thm51}
	Given $\rho\in(0,\infty)$, there exists a Speiser meromorphic function $f$ of order $\rho$ such that
$$\dim\I(f)=\frac{2\rho}{1+\rho}.$$
\end{theorem}

Before we proceed with the proof of this theorem, we show first that how to deduce our Theorem \ref{thm1} and Theorem \ref{thm3} stated in the introduction from this result.

\begin{proof}[Proof of Theorem \ref{thm1}]
	By the above discussions, it suffices to consider the case $d\in(0,2)$. For such $d$, we let $\rho={d}/{(2-d)}$. Then Theorem \ref{thm51} assures the existence of a Speiser function $f$ with order $\rho$ for which the Hausdorff dimension of the escaping set is equal to $d$. The other properties of $f$ stated in Theorem \ref{thm1} follows from our construction.
\end{proof}

\begin{proof}[Proof of Theorem \ref{thm3}]
	By Theorem \ref{equd1} or Theorem \ref{thm42}, for two distinct $\rho_1,\,\rho_2\in (0,2)$ or $(2,\infty)$, there exist two quasiconformally equivalent meromorphic functions $f_1$ and $f_2$ whose orders are respectively $\rho_1$ and $\rho_2$. The above Theorem \ref{thm51} then says that their escaping sets have Hausdorff dimensions $2\rho_i/(1+\rho_i)$, which are different. Since there are uncountably many choices for $\rho_i$, the conclusion follows clearly.
\end{proof}

It remains to prove Theorem \ref{thm51}. The existence of a Speiser function $f$ for a given $\rho$ has been given in Section \ref{CS}. The rest of this subsection is then devoted to checking that the Hausdorff dimension of the escaping set of the constructed function is equal to $2\rho/(1+\rho)$. We will prove this by estimating the Hausdorff  dimension from above and from below in the following. Firstly,
the upper bound follows directly from Theorem 1.1 in \cite{bergweiler2}, i.e. that the Hausdorff dimension is at most $\frac{2 \rho}{1+ \rho}$ ($M=2$ in Theorem 1.1 in \cite{bergweiler2}). Thus it remains to prove the lower bound. Before proving this estimate, we need some preliminaries.

\smallskip
Since $f$ has only finitely many singular values, we may take $R_0>0$ sufficiently large such that $\sing(f^{-1})$ is contained in $D(0,R_0)$. Put $B(R)=\hc\setminus\overline{D}(0,R)$. Then for $R>R_0$ each component of $f^{-1}(B(R))$ is bounded, simply connected and contains exactly one pole of $f$; see \cite[Lemma 2.2]{bergweiler2}. Let $\{a_j\}$ be poles of $f$, arranged in the way such that $\cdots\leq|a_j|\leq|a_{j+1}|\leq\cdots$. Then by discussions in Section \ref{CS} and Theorems \ref{leq2} and \ref{geq2}, all poles are have multiplicity two. From the local behaviour near poles in Section \ref{CS}, we have 
$$f(z)\sim\left(\frac{b_j}{z-a_j}\right)^2\,\,\,\text{~as~}\,\,\,z\to a_j,$$
where
\begin{equation}\label{repole}
|b_j|\sim |a_j|^{1-\rho/2}.
\end{equation}
Denote by $U_j$ the component of $f^{-1}(B(R))$ containing $a_j$. Let $\varphi_j : U_j\to D(0,1/\sqrt{R})$ be a conformal map satisfying $\varphi_{j}(a_j)=0$. Since $|f(z)\varphi_j(z)^2|$ tends to $1$ as $z$ approaches the boundary of $U_j$, and $|f(z)\varphi_j(z)^2|$ is bounded near $a_j$ and not equal to zero in $U_j$. The maximum principle ensures that $|f(z)\varphi_j(z)^2|=1$ and that $|\varphi'_j(a_j)|=1/|b_j|$. 

Since $\varphi_j$ is conformal and $\varphi_{j}(U_j)=D(0,1/\sqrt{R})$, by applying \eqref{one-quarter theorem} of Theorem \ref{koebe} to the inverse function of $\varphi_j$, we can have
$$U_j\supset D\left(a_j,\frac{1}{4}\frac{1}{|\varphi'_{j}(a_j)|\sqrt{R}}\right)=D\left(a_j,\frac{|b_j|}{4\sqrt{R}}\right).$$
Moreover, by choosing $R$ sufficiently large (say, $R>4R_0$), the inverse of $\varphi_j$ can extend to a map which is univalent in $D(0,2/\sqrt{R})$. So with \eqref{estimate of value} of Theorem \ref{koebe} we see that, by putting $\lambda=1/2$,
$$U_{j}\subset D\left(a_j, \frac{\lambda}{(1-\lambda)^2}\frac{1}{|\varphi'_{j}(a_j)|\sqrt{R}}\right)=D\left(a_j, \frac{2|b_j|}{\sqrt{R}}\right).$$
Therefore, we have a good control over the size of $U_j$ in the following sense:
\begin{equation}\label{51}
D\left(a_j,\frac{|b_j|}{4\sqrt{R}}\right)\subset U_j \subset D\left(a_j,\frac{2|b_j|}{\sqrt{R}}\right).
\end{equation}
Moreover, for $z$ in any simply connected domain $D\subset B(R)\setminus\{\infty\}$, by the Monodromy theorem one can define all branches of the inverse of $f$. Denote by $g_j$ an inverse branch of $f$ from $D$ to $U_j$. Then
\begin{equation}\label{52}
\left|g'_{j}(z)\right|\leq C_1\,\frac{|b_j|}{|z|^{3/2}},
\end{equation}
for $z\in D$ and for some constant $C_1>0$. In later estimates, $D$ is usually chosen to be $U_k$ for large $k$.

\bigskip
For $k$ sufficiently large, we have $U_k\subset B(R)$. Then by \eqref{51} and \eqref{52}, one can see that,
$$\diam g_j(U_k)\leq\sup_{z\in U_k}|g'_j(z)|\diam U_k \leq C_1\frac{|b_j|}{|a_k|^{3/2}}\frac{4|b_k|}{\sqrt{R}}.$$
Now suppose that $U_{j_1},\,U_{j_2},\dots, U_{j_{\ell}}$ all are contained in $B(R)$. By the above estimate and induction, we have
\begin{equation}\label{53}
\begin{aligned}
\diam\left(g_{j_1}\circ g_{j_2}\circ\cdots\circ g_{j_{\ell-1}}\right)(U_{j_{\ell}})&\leq C_1\frac{|b_{j_1}|}{|a_{j_2}|^{3/2}}\cdots C_1\frac{|b_{j_{\ell_2}}|}{|a_{j_{\ell-1}}|^{3/2}}\cdot C_1\frac{|b_{j_{\ell-1}}|}{|a_{j_{\ell}}|^{3/2}}\cdot\frac{4|b_{j_{\ell}}|}{\sqrt{R}}\\
&= C_{1}^{\ell-1}\,\frac{4}{\sqrt{R}}\,|b_{j_1}|\,\prod_{k=2}^{\ell}\,\frac{|b_{j_k}|}{|a_{j_k}|^{3/2}}.
\end{aligned}
\end{equation}
In terms of spherical metric, we obtain
\begin{equation}\label{54}
\diam_{\chi}\left(g_{j_1}\circ g_{j_2}\circ\cdots\circ g_{j_{\ell-1}}\right)(U_{j_{\ell}})\leq C_{1}^{\ell-1}\,\frac{32}{\sqrt{R}}\,\prod_{k=1}^{\ell}\,\frac{|b_{j_k}|}{|a_{j_k}|^{3/2}}.
\end{equation}

Now we consider the set of points whose forward orbit always stay in $B(R)$. More precisely, we are looking at a subset of 
$$\j_{R}(f)=\left\{\,z\in B(R):\,f^{n}(z)\in B(R)~\,\text{for all}\,~n\in\mathbb{N}\, \right\}.$$
We also set
$$\I_R(f)=\j_R(f)\cap\I(f).$$
Let also $E_l$ be the collection of all components $V$ of $f^{-l}(B(R))$ for which $f^{k}(V)\subset B(R)$ holds for $0\leq k\leq l-1$. 
Then $E_l$ will be a cover of
$$\left\{\,z\in B(3R):f^{k}(z)\in B(3R)~\,\text{for}\,~0\leq k\leq l-1\,  \right\}.$$

We will use Theorem \ref{mcmlower} in Section \ref{pre} to estimate the lower bound. For $V\in E_l$, there exist $j_1, j_2,\dots,j_{l-1}$ such that
$$f^{k}(V)\subset U_{j_{k+1}}~\,\,\text{for}\,\,~k=0, 1, \dots, l-1.$$
Then by \eqref{res3}, \eqref{res4} and \eqref{54} we have
\begin{equation}\label{diamchiv}
\diam_{\chi}(V)\leq C_{1}^{\ell-1}\,\frac{32}{\sqrt{R}}\,\prod_{k=1}^{\ell}\,\frac{|b_{j_k}|}{|a_{j_k}|^{3/2}}=C_{1}^{\ell-1}\,\frac{32}{\sqrt{R}}\,\prod_{k=1}^{\ell}\,\frac{|B_4|^{\ell}}{|a_{j_k}|^{1/2+\rho/2}}\leq \left(\frac{C_2}{R^{1/2+\rho/2}}\right)^l
\end{equation}
where $C_2>0$ is some constant. In the last inequality we have used the fact that $|a_{j_{k}}|\geq R$. Thus, we put
\begin{equation}\label{dl}
d_l=\left(\frac{C_2}{R^{1/2+\rho/2}}\right)^l.
\end{equation}

To estimate the density of $\overline{E}_{l+1}$ in $V$, we consider for $s\,(\geq R)$ and for large $R$ the following annulus
$$A(s)=\left\{\,z: s<|z|<2s\,\right\}.$$
For $s \geq R$ we have $A(s) \subset B(R)$. Note first that the number of $U_j$'s lying in $A(s)$ is $n(2s, f)-n(s,f)$. (Recall that $n(r,f)$ is the number of poles of $f$ in the disk $\overline{D}(0,r)$). Using \eqref{howmany}, we see that $n(2s, f)-n(s,f)=C_3 s^{\rho}$ with $\rho=\rho(f)$ and certain constant $C_3>0$. By \eqref{51}, \eqref{res3} and \eqref{res4}, for these $U_j$ (contained in $A(s)$),
$$\diam U_j\geq \frac{|b_j|}{2\sqrt{R}}= |B_4|\frac{|a_j|^{1-\rho/2}}{2\sqrt{R}}\geq |B_4|\frac{s^{1-\rho/2}}{2\sqrt{s}}= C_4\,\frac{1}{s^{\rho/2-1/2}},$$
where $C_4=|B_4|/2$. This gives, with $C_5=\pi C_3 C_{4}^2$,
$$\area\left(\overline{E}_1 \cap  A(s)\right)\geq C_3\,s^{\rho}\cdot \pi\left( \frac{C_4}{s^{\rho/2-1/2}}\right)^2={C_5}\,s.$$
So we have
$$ \dens\left(\overline{E}_1, A(s)\right)\geq \frac{C_5}{3\pi s}.$$
Put $E_l^s = E_l \cap f^{-l}((A(s))$. 
Note that, by definition of $V\in E_{l}^s$, there is some $j$ such that $f^{l-1}(V)=U_j$.
By repeated use of Theorem \ref{koebe}, we see that
\begin{equation}\label{wik}
\dens_{\chi}\left(\overline{E}_{l+1}, V \right)\geq \frac{C_7}{R}.
\end{equation}
Here $C_6>0$ is certain constant. Choose $s=2^{k}R$, we have that
$$\dens_{\chi}\left(\overline{E}_{l+1}, V \right)\geq \frac{C_7}{R},$$
where $C_7=C_6/2^k$. Put
\begin{equation}\label{Deltal}
\Delta_l=\frac{C_7}{R}.
\end{equation}

Now with \eqref{dl} and \eqref{Deltal}, we apply Theorem \ref{mcmlower} to obtain that
$$\dim E\geq 2-\limsup_{l\to\infty}\frac{(l+1)\cdot(\log C_7-\log R)}{l\cdot(\log C_2-(\frac{1}{2}+\frac{\rho}{2})\log R)}=2-\frac{\log C_7-\log R}{\log C_2-(\frac{1}{2}+\frac{\rho}{2})\log R}.$$
By taking $R\to\infty$, we obtain
$$\dim E\geq\frac{2\rho}{1+\rho},$$
which implies that
$$\dim\j_{R}(f)\geq\frac{2\rho}{1+\rho}.$$

Now we identify a subset of $\I(f)$ whose Hausdorff dimension gives the right magnitude. To do this, we take an increasing sequence $(R_{k})$ tending to $\infty$ and consider the set of points whose $k$-th iterate falls into $B(R_k)$. Now for each $k$ we define $E_{k}$ as the collection of components $V$ of $f^{-k}(B(R_{k}))$ such that $f^{m}(V)\subset B(R_{k})$ for $0\leq m\leq k-1$. Denote by $\overline{E}_{k}$ the union of the components in $E_{k}$ and put $E=\cap_k\overline{E}_{k}$. It follows that $E$ is a subset of the escaping set $\I(f)$. Now by using similar estimates as before, we can have (compare with \eqref{diamchiv})
$$\diam_{\chi}(V)\leq \prod_{j=1}^{k}\frac{C_8}{R_{j}^{1/2+\rho/2}}$$
and (compare with \eqref{wik})
$$\dens_{\chi}\left(\overline{E}_{k+1}, V \right)\geq \frac{C_9}{R_k}.$$
Here $C_8, C_9>0$ are some constants. So, by putting 
$$d_{k}=\prod_{j=1}^{k}\frac{C_8}{R_{k}^{1/2+\rho/2}}~\,\quad\text{and}\quad\,~\Delta_k=\frac{C_9}{R_k},$$
and using Theorem \ref{mcmlower} we see that
$$\dim E\geq 2-\limsup_{k\to\infty}\frac{(k+1)\log C_9-\sum_{j=1}^{k+1}\log R_j}{k\log C_8-(\frac{1}{2}+\frac{\rho}{2})\sum_{j=1}^{k}\log R_j}$$
Choose a suitable sequence $(R_k)$, say $R_k=e^k$, we obtain
$$\dim\I(f)\geq \dim E\geq \frac{2\rho}{1+\rho}.$$
This completes the proof of Theorem \ref{thm51} and thus Theorem \ref{thm1}.

\subsection{Escaping set of full dimension}\label{efd} 
For having a meromorphic function with an escaping set of full Hausdorff dimension, one can consider a meromorphic function in the class $\s$ with finite order and with one logarithmic singularity over $\infty$. Simple examples can be obtained by considering a meromorphic function $f$ with a polynomial Schwarzian derivative and with more than three singular values. The function $f$ has only finitely many asymptotic values and no critical values (and hence belongs to the class $\s$). Moreover, it has finite order of growth. Assume also that such a function has $\infty$ as an asymptotic value (otherwise we apply a M\"obius map sending one of the asymptotic value to $\infty$ and the resulted function still has finite order by Theorem \ref{fft}). Then the escaping set has full dimension by repeating the argument by Bara\'nski \cite{baranski1} and  Schubert \cite{schubert}. 

However, our main intention here is to find meromorphic functions in class $\s$ for which $\infty$ is \emph{not} an asymptotic value and for which the escaping set has full dimension. By \cite[Theorem 1.1]{bergweiler2}, such a function should either have infinite order and bounded multiplicities for poles or have finite order and unbounded multiplicities for poles. We provide an example with the former property and with full dimension of escaping set. Let $\wp$ be a Weierstra{\ss} elliptic function with respect to a lattice and $c$ is chosen such that it is not a pole of $\wp$. As stated in Proposition \ref{full}, we will consider the function
$$f(z)=\wp(e^z+c).$$
Note that $f$ belongs to the class $\s$. Moreover, $\infty$ is not an asymptotic value of $f$ due to our choice of $c$. All poles of $f$ are double poles.

To prove Proposition \ref{full}, we apply the same method as in the Section \ref{generalcase}; i.e., we use Theorem \ref{mcmlower} to get the lower bound $2$ for $\dim\j_{R}(f)$. And then by taking a sequence $R_k$ tending to infinity we can estimate the Hausdorff dimension of a subset of the escaping set, which is sufficient to get the right lower bound. We will not give a full detailed proof but only address the ideas and difference from the previous case. Let now $z_0$ be a pole of $f$ and the residue of $f$ at $z_0$ is $b(z_0)$. Then by computation using L'Hospital's rule we have
$$|b(z_0)|=\frac{A_1}{|e^{z_0}|}=\frac{A_1}{e^{\re z_0}}$$
for some constant $A_1>0$. We will use same notations as before, in particular, the poles are denoted by $a_j$ arranged in a way such that $\cdots\leq |a_j|\leq |a_{j+1}|\leq\cdots$, and the component of $f^{-1}(B(R))$ containing $a_j$ are denoted by $U_j$. Put $b_j=\res(f,a_j)$. So we have, by using Theorem \ref{koebe} in the same way as before,
$$D\left(a_j, \frac{|b_j|}{4\sqrt{R}}\right)\subset U_j \subset D\left(a_j, \frac{2|b_j|}{\sqrt{R}}\right).$$
Now instead of considering a sequence of annuli as in the lower bound estimate in Section \ref{generalcase}, we consider a sequence of squares symmetric to the positive real axis. More precisely, we consider for $s(\geq R)$ and large $R$ the following squares
$$P(s):=\left\{z=x+iy:\, s<x<2s,\,|y|<\frac{s}{2}\, \right\}.$$
For $s\geq R$ we have $P(s)\subset B(R)$. We need to count the number of poles in $P(s)$. For this purpose, we count first the number of poles in
$$Q(s):=\left\{z=x+i y:\, s<x<2s,\,|y|<\pi \,\right\},$$
which is a subset of $P(s)$. This can be obtained by comparing the area of $\exp(Q(s))$ with that of a parallelogram for the function $\wp$. Denote by $n(P(s),\infty)$ and respectively $n(Q(s),\infty)$ the number of poles (ignoring multiplicities) in $P(s)$ and resp. $Q(s)$. Then, for large $s$,
$$n(Q(s),\infty)=\frac{\pi(e^{4s}-e^{2s})}{C}\sim A_2\,e^{4s}$$
for some constant $A_2>0$. Here $C$ is the area of a fundamental parallelogram of the function $\wp$. So we have for large $s$, by the periodicity of $f$,
$$n(P(s),\infty)\sim A_3\,se^{4s},$$
where $B_3>0$ is a constant.

With the above estimates, by repeating techniques in previous section, we may choose, for some constants $A_4,\,A_5$,
$$d_{l}=\left(\frac{A_4}{e^{R}R^{3/2}}\right)^{l}$$
and
$$\Delta_l=\frac{A_5}{R^2}.$$
Thus by applying Theorem \ref{mcmlower}, we obtain
$$\dim\j_R(f)\geq 2-\limsup_{R\to\infty}\frac{\log A_5 - 2\log R}{\log A_4-R-(3/2)\log R}=2.$$
As noted above, to obtain the Hausdorff dimension of the escaping set of $f$, we need to consider an increasing sequence $(R_k)$ tending to infinity and get similar estimates as before. This is a repeat of the previous argument. We omit details here and finally draw the following conclusion:
$$\dim\I(f)=2.$$

\begin{remark}
	The escaping set of the above function has zero Lebesgue measure by \cite[Theorem 1.3]{bergweiler2}.
\end{remark}

\begin{remark}
It is proved in \cite{kotus2} that if a meromorphic function $f$ is of the form $R(e^z)$, where $R$ is any rational function chosen such that $\infty$ is not an asymptotic value of $f$, then $\dim\I(f)=q/(q+1)<1$ with $q$ the maximal multiplicity of poles. Here the function $\wp(e^z+c)$ suggests that the Hausdorff dimension of the escaping set can be large if one takes transcendental functions instead of rational functions.
\end{remark}

\bigskip
\noindent Magnus Aspenberg

\smallskip
\noindent\emph{Centre for Mathematical Sciences, Lund University, Box 118, 22 100 Lund, Sweden}

\noindent\emph{magnus.aspenberg@math.lth.se}

\bigskip
\noindent Weiwei Cui

\smallskip
\noindent\emph{Shanghai Center for Mathematical Sciences, Fudan University, No. 2005
Songhu Road, Shanghai 200438, China;}

\smallskip
\noindent\emph{Centre for Mathematical Sciences, Lund University, Box 118, 22 100 Lund, Sweden}

\noindent\emph{weiwei.cui@math.lth.se}

\end{document}